\documentclass[a4paper, 11pt]{article}

\usepackage[arxiv]{optional}

\usepackage[english]{babel}
\usepackage[T1]{fontenc}
\usepackage{amsmath}
\opt{arxiv}{%
\usepackage{amsthm}
}%
\usepackage{amsfonts}
\usepackage{amssymb}
\usepackage[final]{graphicx}
\usepackage{paralist}
\usepackage{hyperref}
\usepackage{tikz}
\usepackage{booktabs}

\usepackage{cleveref}
\usepackage{subfigure}
\usepackage{pifont}
\usepackage{caption}
\usepackage{mathtools}

\opt{draft}{%
\usepackage{refcheck}

}

\opt{arxiv}{%
\theoremstyle{plain}
\newtheorem{theorem}{Theorem}[section]
\newtheorem{proposition}[theorem]{Proposition}
\newtheorem{lemma}[theorem]{Lemma}
\newtheorem{corollary}[theorem]{Corollary}
\newtheorem{conject}[theorem]{Conjecture}

\theoremstyle{definition}
\newtheorem{definition}[theorem]{Definition}
\newtheorem{example}[theorem]{Example}
\newtheorem{remark}[theorem]{Remark}
}%
\opt{springer}{%

}%

\newcommand{\eps}{\varepsilon}
\renewcommand{\emptyset}{\varnothing}

\DeclareMathOperator{\im}{Im}
\DeclareMathOperator{\ind}{ind}

\DeclareMathOperator{\re}{Re}

\opt{springer}{%

}
\opt{arxiv}{%

}

\DeclarePairedDelimiter{\abs}{\lvert}{\rvert}

\DeclarePairedDelimiter{\cc}{[}{]} 
\DeclarePairedDelimiter{\co}{[}{[} 

\newcommand{\R}{\ensuremath{\mathbb{R}}}
\newcommand{\C}{\ensuremath{\mathbb{C}}}

\newcommand{\Z}{\ensuremath{\mathbb{Z}}}

\newcommand{\bD}{\mathbb{D}}

\newcommand{\cA}{\mathcal{A}}
\newcommand{\cC}{\mathcal{C}}
\newcommand{\cM}{\mathcal{M}}

\newcommand{\conj}[1]{\overline{#1}}

\newcommand{\wn}[2]{n(#1; #2)}

\DeclareMathOperator{\crit}{crit}

\DeclareMathOperator{\tr}{trace}

\numberwithin{equation}{section}

\title{\textsc{Number and location of pre-images under harmonic mappings in the plane}}

\author{Olivier S\`{e}te\footnotemark[1] \and Jan Zur\footnotemark[1]}

\date{June 29, 2020}

\begin{document}
\maketitle
\renewcommand{\thefootnote}{\fnsymbol{footnote}}
\footnotetext[1]{TU Berlin, Department of Mathematics, MA 3-3, Stra{\ss}e des 
17.\@ Juni 136, 10623 Berlin, Germany. 
\texttt{\{sete,zur\}@math.tu-berlin.de}}

\captionsetup[figure]{skip=0pt}

\begin{abstract}
We derive a formula for the number of pre-images under a non-degenerate 
harmonic mapping $f$, using the argument principle.  This formula reveals a 
connection between the pre-images and the caustics.
Our results allow to deduce the number of pre-images under $f$ 
geometrically for every non-caustic point.  We 
approximately locate the pre-images of points near the caustics.
Moreover, we apply our results to prove that for every 
$k = n, n+1, \ldots, n^2$ there exists a harmonic polynomial of degree $n$ with 
$k$ zeros.
\end{abstract}

\section{Introduction}\label{sect:introduction} 

\begin{figure}[th]
\centering
\includegraphics[width=0.5\linewidth]{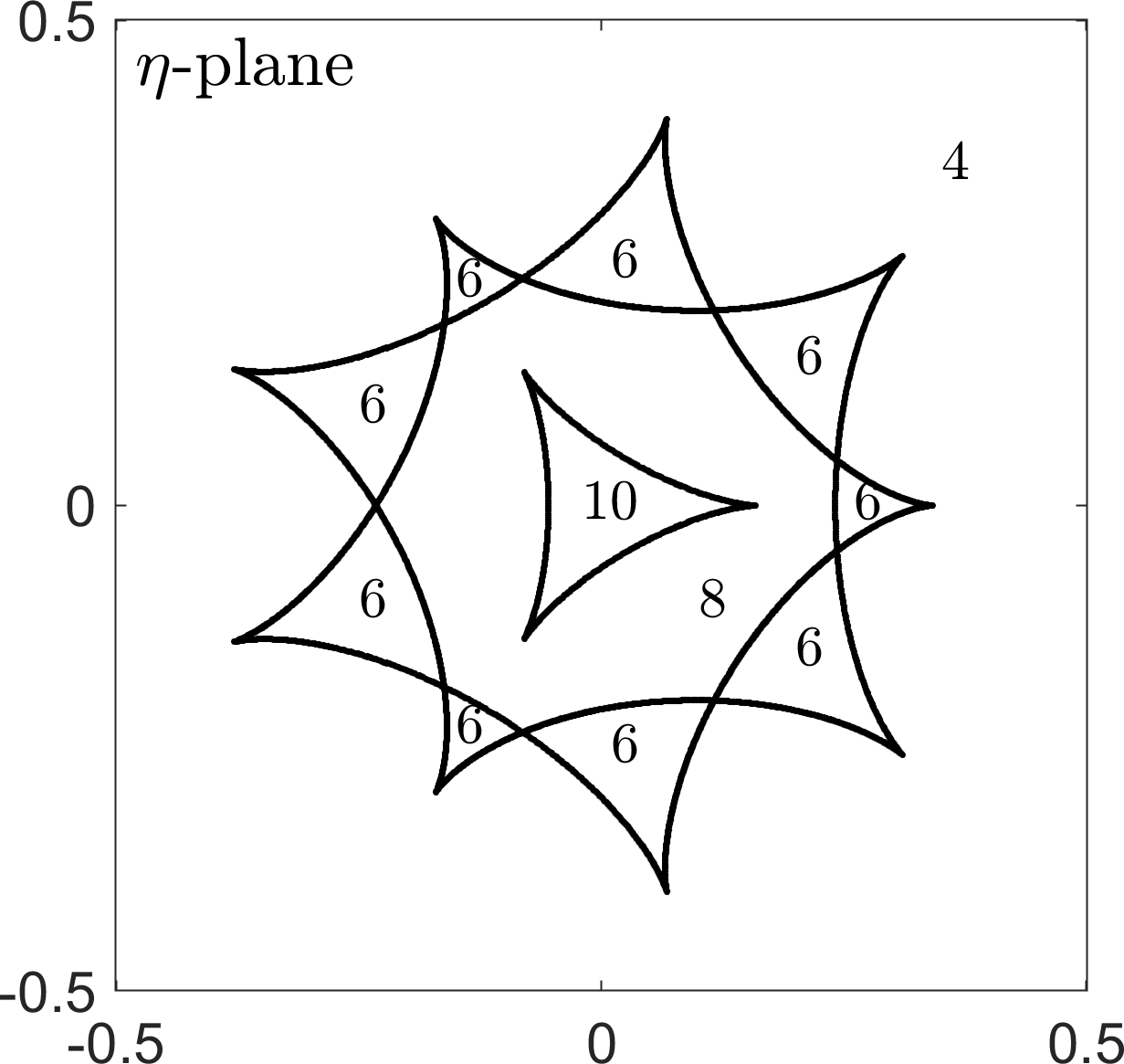}

\vspace{1mm}

\caption{Number of pre-images of $\eta$ under $f(z) = z - \conj{z^2/(z^3 - 
0.6^3)}$ for an $\eta$ in the respective regions; see also 
Example~\ref{ex:mpw} below.  The black lines mark the 
caustics (critical values) of $f$. The number of pre-images of $\eta$ in the 
outer tile corresponds to the number of poles of $f$ (including $\infty$).}
\label{fig:mpw_plain}
\end{figure}

Harmonic mappings in the plane, i.e., functions $f: \Omega \to \C$ with $\Delta 
f = 0$ on an open set $\Omega \subseteq \C$, regained attention in the 
last decades, starting from the seminal work of Clunie and 
Sheil-Small~\cite{ClunieSheil-Small1984}.
See, e.g., the large collection of open problems by Bshouty and 
Lyzzaik~\cite{BshoutyLyzzaik2010} and references therein.
While we consider here multivalent harmonic mappings, also (locally) univalent 
harmonic mappings are of interest, see, e.g., Duren's 
textbook~\cite{Duren2004}, especially in the context of quasi-conformal 
mappings~\cite{Ahlfors1966_qc}.

Numerous authors have studied the number and location of zeros of harmonic 
mappings, i.e., the solutions of $f(z)=0$. Of particular interest have been 
harmonic polynomials of the form $f(z)=p(z)-\conj{z}$ 
\cite{KhavinsonSwiatek2003,Geyer2008}, or $f(z)=p(z) + \conj{q(z)}$ and the 
questions related to Wilmshurst's 
conjecture~\cite{Wilmshurst1998,LeeLerarioLundberg2015,KhavinsonLeeSaez2018}. 
Also, the zeros of rational harmonic mappings of the form $f(z) = r(z)-\conj{z}$ have been 
studied intensively~\cite{KhavinsonNeumann2006,BleherHommaJiRoeder2014,
LuceSeteLiesen2014a,LuceSeteLiesen2014b,LiesenZur2018b}, since these are of 
interest when modeling the phenomenon of gravitational 
lensing~\cite{KhavinsonNeumann2008,Petters2010,BeneteauHudson2018}.

Here we focus on solutions of $f(z) = \eta$ for given (but arbitrary) 
$\eta \in \C$. As shown in~\cite{LiesenZur2018a} for 
rational harmonic mappings of the form $f(z) = r(z) - \conj{z}$, the number of solutions 
can vary significantly under changes of $\eta$.  Moreover, changes only 
occur when $\eta$ is ``moved'' through the caustics of $f$; see 
Figure~\ref{fig:mpw_plain}.
This paper is devoted to study this effect for a more general class of harmonic mappings.
We show the following:

(1) In Section~\ref{sect:global_valence} we derive (local and global) formulas 
for the number of pre-images of $\eta$ under a non-degenerate harmonic mapping 
$f$ (Definition~\ref{def:non-degenerate}) in terms of the poles and the 
winding number of the caustics about $\eta$, e.g.,
\begin{equation}\label{eqn:counting_introduction}
N_\eta(f) = P(f) + 2\sum_{\gamma \in \crit} \wn{f \circ \gamma}{\eta};
\end{equation} 
see Theorem~\ref{thm:counting}.
An immediate consequence of~\eqref{eqn:counting_introduction} is that the number of pre-images changes by $\pm 2$ when $\eta$ changes 
from one side to the other of a single caustic arc; see Figure~\ref{fig:mpw_plain}.

(2) In Section~\ref{sect:local_valence} we complement Lyzzaik's 
study~\cite{Lyzzaik1992} of the local behavior of light harmonic mappings at 
their critical points.
We approximately locate pre-images of $\eta$ near a fold caustic point,
which makes the pre-images also accessible for computations.
Moreover, we determine for which $\eta$ near a fold we have locally two or 
no pre-images; see Theorem~\ref{thm:local_fold}. 

(3) In Section~\ref{sect:polynomials} we apply the results from 
Sections~\ref{sect:global_valence} and~\ref{sect:local_valence} to harmonic 
polynomials.  In particular we prove that for all $k \in \{n,n+1,\ldots,n^2\}$ 
there exists a harmonic polynomial $f(z) = p(z) + \conj{q(z)}$ with $\deg(p) = 
n$ and $\deg(q) < n$ with exactly $k$ zeros, i.e., every number between the 
minimum and maximum can be attained; see Corollary~5.6.  This 
generalizes a result of Bleher et al.~\cite[Thm.~1.1]{BleherHommaJiRoeder2014}.

\section{Preliminaries}
\label{sect:prelim}

The key ingredient to derive the formulas for the exact number of pre-images in 
Section~\ref{sect:global_valence} is the argument principle for harmonic 
mappings, applied on the critical set.
In preparation, we collect and extend several known results in this section.

A \emph{harmonic mapping} is a function $f : \Omega \to \C$ defined on an open 
set $\Omega \subseteq \C$ and with
\begin{equation*}
\Delta f = \partial_{xx} f + \partial_{yy} f
= 4 \partial_{\conj{z}} \partial_z f = 0,
\end{equation*}
where $\partial_z$ and $\partial_{\conj{z}}$ denote the \emph{Wirtinger 
derivatives} of $f$; see e.g.~\cite[Sect.~1.2]{Duren2004}.  If $f$ is harmonic 
in the 
open disk $D = \{z \in \C : \abs{z-z_0} < r\}$, it has a local decomposition
\begin{equation}\label{eqn:local_decomp}
f(z) = h(z) + \conj{g(z)} = \sum_{k=0}^\infty a_k(z-z_0)^k + 
\conj{\sum_{k=0}^\infty b_k(z-z_0)^k}, \quad z \in D,
\end{equation} 
with analytic functions $h$ and $g$ in $D$, which are unique up to 
an additive constant; see~\cite[p.~412]{DurenHengartnerLaugesen1996} 
or~\cite[p.~7]{Duren2004}.
If $f$ is harmonic in the punctured disk $D = \{z \in \C: 0 < \abs{z-z_0} < 
r \}$, it has a local decomposition
\begin{equation}\label{eqn:local_decomp_sing}
f(z)
= \sum_{\mathclap{k=-\infty}}^\infty a_k(z-z_0)^k + 
\conj{\sum_{\mathclap{k=-\infty}}^\infty b_k(z-z_0)^k} + c\log\abs{z-z_0},
\quad z \in D;
\end{equation}
see~\cite{SuffridgeThompson2000,HengartnerSchober1987}.  We consistently 
use the notation from~\eqref{eqn:local_decomp} 
and~\eqref{eqn:local_decomp_sing}.

The \emph{Jacobian} of a harmonic mapping $f$ at $z \in \Omega$ is
\begin{equation} \label{eqn:jacobian}
J_f(z) = \abs{\partial_z f(z)}^2 - \abs{\partial_{\conj{z}} f(z)}^2
= \abs{h'(z)}^2 - \abs{g'(z)}^2,
\end{equation}
where $f = h+\conj{g}$ is a local decomposition~\eqref{eqn:local_decomp}.
We call $f$ \emph{sense-preserving} at $z$ if $J_f(z) > 0$, 
\emph{sense-reversing} at $z$ if $J_f(z) < 0$, and \emph{singular} at $z$ if 
$J_f(z) = 0$.
Moreover, we call $f$ \emph{singular}, if $f$ is singular at one of its zeros.
If $\varphi$ is an analytic function, then  $f \circ \varphi$ is again a 
harmonic mapping and
\begin{equation}\label{eqn:jacobian_transformation}
J_{f \circ \varphi}(z) = J_f(\varphi(z)) \abs{\varphi'(z)}^2.
\end{equation}
In particular, if $\varphi'(z) \neq 0$, the maps $f$ at $\varphi(z)$ and 
$f \circ \varphi$ at $z$ are simultaneously sense-preserving, sense-reversing, 
or singular, respectively.

\subsection{Critical set and caustics}
\label{sect:crit}

The points at which a harmonic mapping $f$ is singular form the 
\emph{critical~set}
\begin{equation} \label{eqn:crit}
\cC = \{ z \in \Omega : J_f(z) = 0 \},
\end{equation}
which consists of the level set of an analytic function, and certain
isolated points, as we see next.

The \emph{second complex dilatation} of a harmonic mapping $f$ is
\begin{equation*}
\omega(z) = \frac{\conj{\partial_{\conj{z}}f(z)}}{\partial_z f(z)}
= \frac{g'(z)}{h'(z)},
\end{equation*}
with the decomposition $f = h + \conj{g}$ 
from~\eqref{eqn:local_decomp};
see~\cite[p.~5]{Duren2004}, \cite[p.~5]{Ahlfors1966_qc} 
or~\cite[p.~71]{SuffridgeThompson2000}.
We assume that $\partial_z f = h'$ has only isolated zeros in $\Omega$, 
so that $\omega$ is analytic in $\{ z \in \Omega : \partial_z f(z) \neq 0 \}$,
and the singularities of $\omega$ in $\Omega$ are poles or removable 
singularities (which we assume to be removed).
Moreover, we assume that $\abs{\omega} \not\equiv 1$ on an open set (harmonic 
mappings with this property are characterized in~\cite[Lem.~2.1]{Lyzzaik1992}).

Let $z_0 \in \Omega$.  If $h'(z_0) \neq 0$, then $J_f(z_0) = \abs{h'(z_0)}^2 - 
\abs{g'(z_0)}^2 = 0$ is equivalent to  $\abs{\omega(z_0)} = 1$, and if $h'(z_0) 
= 0$, then $J_f(z_0) = 0$ is equivalent to $g'(z_0) = 0$.  Hence, 
$\abs{\omega(z_0)} = 1$ implies $J_f(z_0) = 0$, but the converse is not true in 
general.
Define
\begin{equation} \label{eqn:omega_neq_1}
\cM = \{ z \in \cC : \abs{\omega(z)} \neq 1 \}.
\end{equation}
By the above computation,
\begin{equation*}
\cM = \{ z \in \Omega : h'(z) = g'(z) = 0 \text{ and } 
\lim_{\zeta \to z} \abs{\omega(\zeta)} \neq 1 \}.
\end{equation*}
For $z_0 \in \cM$, there exists a neighborhood of $z_0$ containing no other 
point in $\cC$; see~\cite[Lem.~2.2]{Lyzzaik1992}.
By construction,
\begin{equation*}
\cC \setminus \cM
= \{ z \in \Omega : \abs{\omega(z)} = 1 \}
\end{equation*}
is a level set of the analytic function $\omega$.  Hence, $\cC \setminus \cM$ 
consists of analytic curves, which 
intersect in $z_0 \in \cC \setminus \cM$ if and only if $\omega'(z_0) = 
0$.  More precisely, if $\omega^{(k)}(z_0) = 0$ for $k = 1, \ldots, n-1$ and 
$\omega^{(n)}(z_0) \neq 0$, then $2n$ analytic arcs meet at $z_0$ with 
equispaced angles~\cite[p.~18]{Walsh1950}; see also 
Example~\ref{ex:wilmshurst}.

At points $z \in \cC \setminus \cM$ with $\omega'(z) \neq 0$, the equation
\begin{equation} \label{eqn:parametrization}
\omega(\gamma(t)) = e^{it}
\end{equation}
implicitly defines a local analytic parametrization $z = \gamma(t)$ of $\cC 
\setminus \cM$.
We can write it locally as $\gamma(t) = \omega^{-1}(e^{it})$
with a continuous branch of $\omega^{-1}$.
The corresponding tangent vector at $z = \gamma(t)$ is
\begin{equation} \label{eqn:tangent_to_crit}
\gamma'(t)
= i \frac{\omega(z)}{\omega'(z)}.
\end{equation}
By construction $f$ is sense-preserving to the left of $\gamma$, and 
sense-reversing to the right of $\gamma$.

The image of the critical set under a harmonic mapping $f$ plays a decisive 
role for the number of pre-images.
We call the set of critical values of $f$, i.e., $f(\cC)$, the set of 
\emph{caustic points}, or simply the \emph{caustics} of $f$.
An $\eta \in \C$ has a pre-image under $f$ on the critical set if, and 
only if, $\eta$ is a caustic point.

The next lemma characterizes a tangent vector to the caustics and the curvature 
of the caustics; see~\cite[Lem.~2.3]{Lyzzaik1992}.

\begin{lemma} \label{lem:tangent_to_caustic}
Let $f$ be a harmonic mapping, $z_0 \in \cC \setminus \cM$ with $\omega'(z_0) 
\neq 0$, and let $z_0 = \gamma(t_0)$ with the 
parametrization~\eqref{eqn:parametrization}.
Then $f \circ \gamma$ is a parametrization of a caustic and the
corresponding tangent vector at $f(z_0)$ is
\begin{equation*}
\tau(t_0) = \frac{d}{dt} (f \circ \gamma)(t_0)
= e^{- i t_0/2} \psi(t_0),
\end{equation*}
with
\begin{equation*}
\psi(t_0) = 2 \re( e^{i t_0/2} h'(\gamma(t_0)) \gamma'(t_0)),
\end{equation*}
where $f = h + \conj g$ is a decomposition~\eqref{eqn:local_decomp} in a 
neighborhood of $z_0$.  In particular, the rate of change of the argument of 
the 
tangent vector is
\begin{equation*}
\frac{d}{dt} \arg(\tau(t)) \big\vert_{t = t_0} = - \frac{1}{2}
\end{equation*}
at points where $\psi(t_0) \neq 0$, i.e., the curvature of the 
caustics is constant with respect to the 
parametrization $f \circ \gamma$.

Moreover, $\psi$ has either only finitely many zeros, or is identically zero, 
in which case $f$ is constant on $\gamma$.
\end{lemma}

\begin{definition}
In the notation of Lemma~\ref{lem:tangent_to_caustic}, assume that the tangent 
$\tau(t_0)$ exists.  Then, the point $(f \circ \gamma)(t_0)$ is called
\begin{enumerate}
\item a \emph{fold caustic point} or simply a \emph{fold}, if the tangent is 
non-zero,

\item a \emph{cusp} of the caustic, if $\psi$ has a zero with a sign change at 
$t_0$.
\end{enumerate}
\end{definition}

\begin{remark}
\begin{enumerate}
\item If $(f \circ \gamma)(t_0)$ is a fold, then $f$ 
is \emph{light} (i.e., $f^{-1}( \{ \eta \} )$ is empty~or totally disconnected 
for every $\eta \in \C$) in a neighborhood of $z_0 = \gamma(t_0)$.
Indeed, if $\cC \setminus \cM$ can be
parametrized according to~\eqref{eqn:parametrization}, then $J_f$ is not 
identically zero.
Also, $f \circ \gamma$ is not constant at a fold.
Hence, $f$ is light in a neighborhood of $z_0$ by~\cite[Thm.~2.1]{Lyzzaik1992}.

\item
At a cusp, 
the tangent vector becomes zero and
the argument of the tangent vector jumps by $+\pi$.
Note that the caustic either has only a finite number of cusps, or degenerates 
to a single point by Lemma~\ref{lem:tangent_to_caustic}.

\item 
In~\cite[Def.~2.2]{Lyzzaik1992}, a critical point $z_0 = \gamma(t_0)$ is called 
a critical point of
(i) the first kind, if $f(z_0)$ is a cusp,
(ii) the second kind, if $h'(z_0) = 0$ or $g'(z_0) = 0$, and if $\psi(t_0) = 0$ 
but $\psi$ does not change its sign, and
(iii) the third kind, if $\omega'(z_0) = 0$.
\end{enumerate}
\end{remark}

The curvature and the cusps of the caustics of $f$ are apparent in the examples 
in Figure~\ref{fig:mpw_log}.
The next lemma characterizes the fold caustic points in terms of the coefficients in~\eqref{eqn:local_decomp}.

\begin{lemma}\label{lem:cusp_condition}
Let $f$ be a harmonic mapping, $z_0 \in \cC \setminus \cM$ with $\omega'(z_0) 
\neq 0$ and $h'(z_0) \neq 0$, and let $z_0 = \gamma(t_0)$ with the 
parametrization~\eqref{eqn:parametrization}.
We consider the decomposition~\eqref{eqn:local_decomp} of $f$ at $z_0$ and
define $\theta \in \co{0, \pi}$ by $\conj{b}_1 = a_1 e^{i 2 \theta}$.
Then the following are equivalent:
\begin{enumerate}
\item $\psi(t_0) \neq 0$,
\item $\displaystyle \im \left( \frac{1}{e^{i t_0/2} a_1} \left( 
\frac{a_2}{a_1} - \frac{b_2}{b_1} \right) \right) \neq 0$,
\item $\displaystyle \im \left( \frac{a_2}{a_1} e^{i \theta} + 
\conj{\left(\frac{b_2}{b_1} e^{i \theta} \right)} \right) \neq 0$.
\end{enumerate}
\end{lemma}

\begin{proof}
Using~\eqref{eqn:tangent_to_crit}, $e^{it_0} = \omega(z_0) = b_1/a_1$ and 
$\omega'(z_0) = 2 \frac{b_2 a_1 - b_1 a_2}{a_1^2}$, we have
\begin{equation*}
0 \neq \psi(t_0)
= 2 \re \left( e^{i t_0/2} h'(z_0) i \frac{\omega(z_0)}{\omega'(z_0)} \right)
= \re \left( i e^{i t_0/2} a_1 \frac{b_1 a_1}{b_2 a_1 - b_1 a_2} \right).
\end{equation*}
Since $\re(z) \neq 0$ if and only if $\re(1/z) \neq 0$ (for $z \neq 0$), this 
is equivalent to
\begin{equation*}
0 \neq \re \left( -i \frac{1}{e^{i t_0/2} a_1} \frac{b_2 a_1 - b_1 a_2}{b_1 
a_1} 
\right)
= - \im \left( \frac{1}{e^{i t_0/2} a_1} \left( \frac{a_2}{a_1} - 
\frac{b_2}{b_1} \right) \right).
\end{equation*}
Write $a_1 = \abs{a_1} e^{i \alpha}$, then $b_1 = a_1 e^{it_0} = 
\conj{a}_1 e^{-i 2 \theta}$ implies $e^{i (2 \alpha + t_0)} = e^{- i 2 
\theta}$, and hence $e^{i t_0/2} a_1 = \pm \abs{a_1} e^{- i \theta}$,
which yields the equivalence of 2.\@ and 3.
\end{proof}

\subsection{The argument principle for harmonic mappings}

Let $f$ be continuous and non-zero on the trace of a curve 
$\gamma : \cc{a, b} \to \C$.
Then the 
\emph{winding of $f$ on $\gamma$} is defined as the change of argument of 
$f(z)$ as $z$ travels along $\gamma$ from $\gamma(a)$ to $\gamma(b)$, divided 
by $2\pi$, i.e.,
\begin{equation} \label{eqn:winding}
W(f; \gamma) = \frac{1}{2 \pi} \Delta_\gamma \arg(f(z)) = \frac{1}{2 \pi} 
(\theta(b) - \theta(a)),
\end{equation}
where $\theta : \cc{a, b} \to \R$ is continuous with $\theta(t) = 
\arg(f(\gamma(t))$;
see~\cite[Sect.~2.3]{Balk1991} or~\cite[Ch.~7]{Beardon1979} for details.

Let now $\gamma$ be a closed curve.
We denote the \emph{winding number} of $\gamma$ about $\eta \in 
\C \setminus \tr(\gamma)$ by $\wn{\gamma}{\eta}$,
which is related to the winding through
\begin{equation}\label{eqn:winding_winding_number}
W(f;\gamma) = \wn{f\circ\gamma}{0} \quad \text{and} \quad
\wn{\gamma}{\eta} = W(z \mapsto z-\eta; \gamma).
\end{equation}
In particular, $W(f; \gamma)$ is an integer.
Note that $W(f; \gamma) = \wn{f\circ\gamma}{0} = 0$ if $f$ is constant on 
$\gamma$.
Moreover, the winding is also called the degree or 
topological degree of $f$ on $\gamma$; 
see~\cite[p.~3]{Lloyd1978} or~\cite[p.~29]{Sheil-Small2002}.

The argument principle for a continuous function $f$ 
relates the winding of $f$ to the indices of its exceptional points.
A point $z_0 \in \C$ is called an \emph{isolated exceptional point} of a 
function $f$, if $f$ is continuous and non-zero in a punctured neighborhood 
$D = \{ z \in \C : 0 < \abs{z - z_0} < r \}$ of $z_0$, and if
$f$ is either zero, not continuous, or not defined at $z_0$.
Then the \emph{Poincar\'e index} of $f$ at $z_0$ is defined as 
\begin{equation}
\ind(f;z_0) = W(f;\gamma),
\end{equation}
where $\gamma$ is a closed Jordan curve in $D$ about $z_0$ oriented in the 
positive sense, i.e., with $\wn{\gamma}{z_0} = 1$.
The Poincar\'e index is also called the \emph{index}~\cite[Def.~2.2.2]{Lloyd1978} or 
the multiplicity~\cite[p.~44]{Sheil-Small2002}.
Similarly, $\infty$ is an isolated exceptional point of $f$, if $f$ is 
continuous and non-zero in $D = \{ z \in \C : \abs{z} > R \}$.  We define 
$\ind(f; \infty) = W(f; \gamma)$, where $\gamma$ is a 
closed Jordan curve in $D$ which is negatively oriented and surrounding the 
origin, such that $\infty$ lies on the left of $\gamma$ on the Riemann sphere
$\widehat{\C} = \C \cup \{ \infty \}$.
In either case the Poincar\'e index is independent of the choice of $\gamma$.
We get with $\varphi(z) = z^{-1}$
\begin{equation}\label{eqn:index_infty_phi}
\ind(f; \infty) = W(f; \gamma) = W(f \circ \varphi; \varphi^{-1} \circ \gamma)
= \ind(f \circ \varphi; 0).
\end{equation}
The Poincar\'e index generalizes the multiplicity of zeros 
and order of poles 
of an analytic function; see e.g.~\cite[p.~44]{Sheil-Small2002}.

The following version of the argument principle for continuous functions can be 
obtained from~\cite[Sect.~2.3]{Balk1991}, or~\cite[Sect.~2.3]{Sheil-Small2002}. 
Special versions for harmonic mappings are given 
in~\cite{DurenHengartnerLaugesen1996} 
and~\cite[Thm.~2.2]{SuffridgeThompson2000}.

\begin{theorem}[Argument principle] \label{thm:winding}
Let $D$ be a multiply connected domain in $\widehat{\C}$ whose boundary 
consists of Jordan 
curves $\gamma_1, \ldots, \gamma_n$, which are oriented such that $D$ is on the 
left.
Let $f$ be continuous and non-zero in $\overline{D}$, except for finitely many 
exceptional points $z_1,\dots,z_k \in D$.  We then have
\begin{equation*}
\sum_{j=1}^n W(f; \gamma_j) = \sum_{j=1}^k \ind(f; z_j).
\end{equation*}
\end{theorem}

Using the argument principle and the definition of the Poincar\'e index at 
infinity yields the following theorem.

\begin{theorem} \label{thm:sum_of_all_indices}
Let $f$ be defined, continuous and non-zero on $\widehat{\C}$,
except for finitely many isolated exceptional points $z_1, \ldots, z_n$ in
$\widehat{\C}$, then
\begin{equation*}
\sum_{j=1}^n \ind(f; z_j) = 0.
\end{equation*}
\end{theorem}

The exceptional points of a harmonic mapping $f$ are its zeros and 
points where $f$ is not defined.
We determine their indices, beginning with the zeros; 
see~\cite[p.~413]{DurenHengartnerLaugesen1996} 
or~\cite[p.~66]{SuffridgeThompson2000}.

\begin{proposition}\label{prop:poincare_index_zeros}
Let $f$ be a harmonic mapping with a zero $z_0$, such that the local 
decomposition~\eqref{eqn:local_decomp} is of the form
\begin{equation*}
f(z) = \sum_{k=n}^\infty a_k (z-z_0)^k + \conj{\sum_{k=n}^\infty b_k 
(z-z_0)^k}, \quad n \geq 1,
\end{equation*}
where $a_n$ or $b_n$ can be zero, then
\begin{equation} \label{eqn:index_zero_in_N}
\ind(f; z_0) = \begin{cases}
+n &\text{ if } \abs{a_n} > \abs{b_n}, \\
-n &\text{ if } \abs{a_n} < \abs{b_n},
\end{cases}
\end{equation}
and, in particular,
\begin{equation} \label{eqn:index_nonsingular_zero}
\ind(f; z_0) = \begin{cases}
+1 &\text{ if $f$ is sense-preserving at } z_0, \\
-1 &\text{ if $f$ is sense-reversing at } z_0.
\end{cases}
\end{equation}
\end{proposition}

A zero $z_0$ of a harmonic 
mapping $f$ with $\ind(f; z_0) \in \Z\setminus\{-1,1\}$ is a singular zero by 
the above result.
Proposition~\ref{prop:poincare_index_zeros} covers non-singular zeros and the 
zeros in $\cM$; see~\eqref{eqn:omega_neq_1}.
If $\abs{a_n} = \abs{b_n} \neq 0$, then $z_0$ is a singular 
zero in $\cC \setminus \cM$,
in which case the determination of the index is more challenging;
see~\cite{LuceSete2017} for the special case $f(z) = h(z) - \conj{z}$.

\begin{remark} \label{rem:zeros_in_M}
Zeros of $f$ in $\cM$ can be interpreted as multiple zeros of $f$.
For a zero $z_0 \in \cM$ of $f$, there 
exists $r > 0$ such that $f$ is defined, non-zero and either sense-preserving 
or sense-reversing in $D = \{ z \in \C : 0 < \abs{z-z_0} \leq r \}$.
For $0 < \abs{\eps} < m = \min_{\abs{z-z_0} = r} \abs{f(z)}$ and $z$ with
$\abs{z - z_0} = r$ we have
\begin{equation*}
\abs{f(z) + \eps - f(z)} = \abs{\eps} < m \leq \abs{f(z)},
\end{equation*}
which implies $W(f + \eps; \gamma) = W(f; \gamma) = \ind(f; z_0)$ by Rouch\'e's 
theorem; see e.g.~\cite[Thm.~2.3]{SeteLuceLiesen2015a}.
Since $f + \eps$ has no poles in $\overline{D}$ and $f(z_0) + \eps \neq 0$, it 
has $\abs{\ind(f; z_0)}$ many distinct zeros in $D$ by the 
argument principle.
\end{remark}

Isolated exceptional points where $f$ is not defined are classified according 
to the limit $\lim_{z \to z_0} f(z)$; 
see~\cite[Def.~2.1]{SuffridgeThompson2000}, \cite[p.~44]{Sheil-Small2002}, and 
the classical notions for real-valued harmonic functions, e.g.~\cite[\S 
15.3, III]{Henrici1986}.

\begin{definition}\label{def:pole}
Let $f$ be a harmonic mapping in a punctured disk around $z_0 \in \C$.
Then $z_0$ is called
\begin{compactenum}
\item a \emph{removable singularity} of $f$, if $\lim_{z \to z_0} f(z) = c \in 
\C$,
\item a \emph{pole} of $f$, if $\lim_{z \to z_0} f(z) = \infty$,
\item an \emph{essential singularity} of $f$, if $\lim_{z \to z_0} f(z)$ does 
not exist.
\end{compactenum}
\end{definition}

If one defines $f(z_0) = c$ at a removable singularity, then $f$ is harmonic 
in~$z_0$; apply~\cite[Thm.~15.3d]{Henrici1986} to the real and imaginary parts 
of $f$.  In the sequel, we assume that removable singularities have been 
removed.
If $c = 0$, then $z_0$ is a zero of $f$, and still an exceptional point.

For most poles of harmonic mappings, the Poincar\'e index can be 
determined from the decomposition~\eqref{eqn:local_decomp_sing}.

\begin{proposition}\label{prop:poincare_index_poles}
Let $f$ be a harmonic mapping in a punctured neighborhood of $z_0$, such that
the local decomposition~\eqref{eqn:local_decomp_sing} is of the form
\begin{equation*}
f(z) = \sum_{\mathclap{k=-n}}^\infty a_k (z-z_0)^k
+ \conj{\sum_{\mathclap{k=-n}}^\infty b_k (z-z_0)^k} + c \log\abs{z-z_0},
\end{equation*}
where $a_{-n}$ or $b_{-n}$ can be zero, then
\begin{equation*}
\ind(f; z_0) = \begin{cases}
-n & \text{ if } n \geq 1 \text{ and } \abs{a_{-n}} > \abs{b_{-n}}, \\
+n & \text{ if } n \geq 1 \text{ and } \abs{a_{-n}} < \abs{b_{-n}}, \\
\phantom{+}0 & \text{ if } n = 0 \text{ and } c \neq 0.
\end{cases}
\end{equation*}
Moreover, in each case $z_0$ is a pole of $f$.
In the first case, $f$ is sense-preserving near $z_0$, and in the second it is 
sense-reversing near $z_0$.
In the third case, $z_0$ is an accumulation point of the critical set of $f$.
\end{proposition}

\begin{proof}
See~\cite[Lem.~2.2, 2.3, 2.4]{SuffridgeThompson2000} for the first two cases.
In the third case, we have $\ind(f; z_0) = 0$
by~\cite[pp.~70--71]{SuffridgeThompson2000}.
Moreover, $\omega$ can be continued analytically to $z_0 \notin \Omega$ with
$\abs{\omega(z_0)} = \lim_{z \to z_0} \abs{\omega(z)} = 1$, since
$\partial_z f(z) = \frac{c}{2} \frac{1}{z-z_0} + \sum_{k=1}^\infty a_k k 
(z-z_0)^{k-1}$ and $\partial_{\conj{z}} f(z) = \frac{\conj{c}}{2} 
\frac{1}{z-z_0} + \sum_{k=1}^\infty b_k k (z-z_0)^{k-1}$.
Hence $z_0$ is an accumulation point of the critical set of $f$ by the maximum 
modulus principle for $\omega$.
\end{proof}

\begin{remark} \label{rem:singularities_on_crit}
If $n \geq 1$ and $\abs{a_{-n}} = \abs{b_{-n}} \neq 0$, we have that:
\begin{enumerate}
\item $z_0$ is an accumulation point of the critical set of $f$, as in the 
proof,

\item $z_0$ is a pole or an essential singularity of $f$, and both cases occur.
Consider $f_1(z) = z^{-2} + z^{-1} + \conj{z}^{-2}$ and
$f_2(z) = z^{-2} + z + \conj{z}^{-2}$, for which $z_0 = 0$ is an isolated 
exceptional point.
The origin is a pole of $f_1$, since $\lim_{z \to 0} f_1(z) = \infty$, and
$\ind(f_1; 0) = 0$; see~\cite[Ex.~2.6]{SuffridgeThompson2000}.
In contrast, $\lim_{z \to 0} f_2(z)$ does not exist (compare the limits on the 
real axis and the lines with $\re(z^{-2}) = 0$), i.e., $f_2$ has an essential 
singularity at $0$.
\end{enumerate}
\end{remark}

\section{The number of pre-images}\label{sect:global_valence} 

For non-degenerate harmonic mappings $f$, 
we derive explicit formulas for the number of pre-images
of 
a non-caustic point $\eta$, in terms of the poles of $f$
and of the winding 
number of the caustics of $f$ about $\eta$. The proofs are based on the argument principle.  
Moreover, we deduce geometrically the number of pre-images from the caustics.

\begin{definition}\label{def:non-degenerate}
We call a harmonic mapping $f$ \emph{non-degenerate}, if the following 
conditions hold:
\begin{enumerate}
\item $f$ is defined in $\widehat{\C}$ with the possible exception of finitely 
many 
poles,

\item at a pole $z_0 \in \C$ of $f$, the 
decomposition~\eqref{eqn:local_decomp_sing} 
has the form
\begin{equation} \label{eqn:near_pole}
f(z) = \sum_{\mathclap{k = -n}}^\infty a_k(z-z_0)^k + \conj{\sum_{\mathclap{k = 
-n}}^\infty b_k(z-z_0)^k} + c\log\abs{z-z_0},
\end{equation}
with $n \geq 1$ and $\abs{a_{-n}} \neq \abs{b_{-n}}$.
And if $\infty$ is a pole of $f$, then 
\begin{equation} \label{eqn:near_pole_infty}
f(z) = \sum_{\mathclap{k = -\infty}}^n a_k z^k + \conj{\sum_{\mathclap{k = 
-\infty}}^n b_k z^k} + c \log \abs{z}, \quad \text{for } \abs{z} > R,
\end{equation}
with $n \geq 1$ and $\abs{a_n} \neq \abs{b_n}$, and $R > 0$,

\item the critical set $\cC$ of $f$ is bounded.
\end{enumerate}
\end{definition}

\begin{remark}
\begin{enumerate}
\item 
Item~1 in Definition~\ref{def:non-degenerate} allows to apply the 
argument principle globally.
By~2., we can determine the Poincar\'e index of a pole with 
Proposition~\ref{prop:poincare_index_poles}, and the poles are not accumulation 
points of $\cC$; see Remark~\ref{rem:singularities_on_crit}.
In particular, $\cC$ is a closed subset of $\C$.

\item 
Harmonic polynomials $f(z) = p(z) + \conj{q(z)}$ with $\deg(p) > \deg(q)$, and 
rational harmonic mappings $f(z) = r(z) - \conj{z}$ with $\lim_{z \to \infty} 
f(z) = \infty$ are non-degenerate.
For these functions, the number of zeros or pre-images is intensively studied; 
see e.g.~\cite{Wilmshurst1998, KhavinsonSwiatek2003, KhavinsonNeumann2006, 
Geyer2008,BleherHommaJiRoeder2014, LuceSeteLiesen2014a, LuceSeteLiesen2014b, 
SeteLuceLiesen2015a,LeeLerarioLundberg2015, LiesenZur2018a, LiesenZur2018b, 
BeneteauHudson2018, KhavinsonLeeSaez2018}.

\item
We discuss the difference between non-degenerate harmonic mappings and the maps 
in~\cite{Lyzzaik1992, Neumann2005}.
By~\cite[Thm.~2.1]{Lyzzaik1992}, a harmonic mapping is either (a) light,
(b) has a zero Jacobian, or (c) is constant on an analytic subarc of
$\cC \setminus \cM$.
While Lyzzaik~\cite{Lyzzaik1992} and Neumann~\cite{Neumann2005} consider
harmonic mappings that are light (case (a)) and have no poles,
we allow cases (a) and (c) and certain poles.
For example, the harmonic mapping $f(z) = \frac{1}{z} - \conj{z}$,
modeling the Chang-Refsdal lens in gravitational lensing~\cite{AnEvans2006},
is non-degenerate with poles at $0$ and $\infty$, and with critical set $\cC = 
\{ z \in \C : \abs{z} = 1 \}$.  It is not light, since $f(\cC) = \{ 0 \}$.

\item It is possible that different arcs of the critical set are mapped onto 
the same caustic arc; see Example~\ref{ex:multiple_caustic_arc}.
\end{enumerate}
\end{remark}

\subsection{A formula for the number of pre-images}
\label{sect:number_of_preimages}

To count the number of pre-images under $f$ with the argument principle, we 
separate the regions where $f$ is sense-preserving and sense-reversing.

Let $f$ be a non-degenerate harmonic mapping.  In particular, the critical set 
$\cC$ is bounded and closed.
For each connected component $\Gamma$ of $\cC \setminus \cM$, we construct a 
single 
closed curve $\gamma$ parametrizing $\Gamma$ and traveling through every 
critical arc exactly once, according to~\eqref{eqn:parametrization}.
There are two possibilities.
\begin{enumerate}
\item If $\omega'$ is non-zero on $\Gamma$, then $\Gamma$ is the trace of a 
closed Jordan curve $\gamma$.
\item If $\omega'$ has zeros on $\Gamma$, then $\Gamma$ consists of Jordan 
arcs that meet at the zeros of $\omega'$, and we proceed as follows.
We interpret the component $\Gamma$ as a directed multigraph with 
intersection points as vertices and critical arcs as arcs of the graph, 
directed in the sense of~\eqref{eqn:parametrization}.
At a vertex corresponding to an $(n-1)$-fold zero of $\omega'$, $2n$ arcs meet.
Due to the orientation of the arcs, the same number of arcs are incoming 
and outgoing.
Hence we find an Euler circuit in the graph~\cite[Sect.~I.3]{Bollobas1998},
which corresponds to the desired parametrization $\gamma$ of $\Gamma$.
\end{enumerate}
We call the above $\gamma$ a \emph{critical curve}, and denote the set of all 
these curves by $\crit$; see Figure~\ref{fig:mpw_log} below for examples.

The critical set induces a partition of $\widehat{\C} \setminus \cC$ into open 
and connected components $A$, where $\partial A \subseteq \cC$ and 
$f$ is either sense-preserving or sense-reversing on $A$
(more precisely on $A$ minus the poles of $f$).
Such a component may or may not be simply connected; see 
Figure~\ref{fig:mpw_log} (top left).
Denote the component containing $\infty$ by $A_\infty$.
For $A \neq A_\infty$, note that $\omega$ has at least one zero/pole in $A$ if 
$f$ is sense-preserving/sense-reversing in $A$, by the minimum modulus 
principle/maximum modulus principle for $\omega$.
If $\omega$ is identically zero/infinity, then $f$ is analytic/anti-analytic, 
and there is only one 
component.  Otherwise, $\omega$ has only finitely many zeros and poles 
on the compact set 
$\widehat{\C} \setminus A_\infty$, and there are at most finitely many other 
components,
and we write
\begin{equation} \label{eqn:setA}
\cA = \{ A_1,\dots, A_m\}.
\end{equation}
This generalizes a similar partition for rational harmonic mappings of the form
$f(z) = r(z) - \conj{z}$ from~\cite[Sect.~2]{LiesenZur2018a}.

For $A \in \cA$, we construct parametrizations $\gamma_1,\ldots,\gamma_n$ 
according to~\eqref{eqn:parametrization} of the connected components 
$\Gamma_1,\ldots,\Gamma_n$ of $\Gamma = (\partial A)\setminus \cM$.
If $\omega'$ is non-zero on $\Gamma_j$, then there exists a closed Jordan curve 
$\gamma_j$ with $\tr(\gamma_j) = \Gamma_j$ as before.
Otherwise we interpret $\Gamma_j$ as a directed multigraph and show the 
existence of an Euler circuit as above.
For a zero $z_0 \in \Gamma_j$ of $\omega'$ the set $A_\eps = \{z \in A : 0 < 
\abs{z - z_0} < \eps\}$ consists of $k$ connected components for $\eps >  0$ 
sufficiently small.
Every component of $A_\eps$ produces one ingoing and one outgoing arc at the 
vertex corresponding to $z_0$; see 
Figure~\ref{fig:selfintersection_of_critical_set} (left). 
Hence, there exists an Euler circuit in $\Gamma_j$ and we denote by $\gamma_j$ 
a parametrization according to~\eqref{eqn:parametrization} of this circuit.
Applying the above construction to all $A \in \cA$ yields not necessarily a 
disjoint partition of $\cC\setminus\cM$, see Figure~\ref{fig:mpw_log} (bottom 
left), and hence cannot be used in Theorem~\ref{thm:counting}.
In particular $\gamma_j$ is potentially not a critical curve.

\begin{figure}[t]
{\centering
\usetikzlibrary{decorations.markings}
\begin{tikzpicture}[declare function = {
r=2;
rmin = 0.4;
alpha = 60; 
},
decoration={markings, mark=at position 0.5 with {\arrow{>}}}]

\foreach \phi in { 0, 240 }
\fill[color = lightgray] (0,0) -- (\phi:r) arc (\phi:\phi+alpha:r) -- cycle;

\foreach \phi in { 0, 240 }
\draw[thick, postaction={decorate}] (0,0) -- (\phi:r);
\foreach \phi in { 60, 300 }
\draw[thick, postaction={decorate}] (\phi:r) -- (0,0);

\draw[thick, dashed, postaction={decorate}] (0,0) -- (120:r);
\draw[thick, dashed, postaction={decorate}] (180:r) -- (0,0);

\foreach \phi in { 30, 150, 270 }
\node[] at (\phi:1.3) {$+$};
\foreach \phi in { 90, 210, 330 }
\node[] at (\phi:1.3) {$-$};

\draw[] (0,0) circle (r);
\draw[fill=black] (0,0) circle (0.2em);
\node[above] at (0, 0.15) {$z_0$};

\begin{scope}[shift = {(7, 0)}]
\fill[color = lightgray] (0,0) -- (0:r) arc (0:60:r) -- cycle;
\fill[color = lightgray] (0,0) -- (240:r) arc (240:300:r) -- cycle;

\draw[thick, dashed, postaction={decorate}] (0,0) -- (120:r);
\draw[thick, dashed, postaction={decorate}] (180:r) -- (0,0);

\draw[thick, postaction={decorate}] (60:r) -- (60:rmin);
\draw[thick] (60:rmin) -- (0:rmin);
\draw[thick, postaction={decorate}] (0:rmin) -- (0:r);

\draw[thick, postaction={decorate}] (300:r) -- (300:rmin);
\draw[thick] (300:rmin) -- (240:rmin);
\draw[thick, postaction={decorate}] (240:rmin) -- (240:r);

\foreach \phi in { 30, 150, 270 }
\node[] at (\phi:1.3) {$+$};
\foreach \phi in { 90, 210, 330 }
\node[] at (\phi:1.3) {$-$};

\draw[] (0,0) circle (r);
\draw[fill=black] (0,0) circle (0.2em);
\node[above] at (0, 0.15) {$z_0$};
\end{scope}
\end{tikzpicture}
\vspace{2mm}

}
\caption{Left: $A_\eps$ (shaded) and oriented critical arcs near a zero $z_0$ 
of $\omega'$.
Right: Deformation of $\gamma_j$ in the proof of 
Theorem~\ref{thm:counting_on_tile}.  The $+/-$ signs indicate regions where $f$ 
is sense-preserving/sense-reversing.}
\label{fig:selfintersection_of_critical_set}
\end{figure}
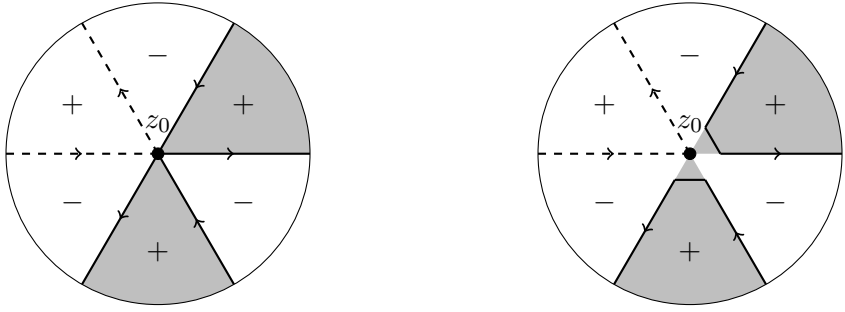

We determine the number of pre-images in one component $A \in \cA$.

\begin{theorem}\label{thm:counting_on_tile}
Let $f$ be a non-degenerate harmonic mapping, $A \in \cA$, and let 
$\gamma_1, \ldots, \gamma_n$ be a parametrization of $\Gamma = (\partial A) 
\setminus \cM$ as above.
Moreover, let  $z_1, \ldots, z_k$ be the poles of $f$ in $A$, and define $P(f; 
A) = \sum_{j=1}^k \abs{\ind(f; z_j)}$.
Then, for $\eta \in \C$ such that $f - \eta$ is non-zero on $\partial A$, the 
number $N_\eta(f; A)$ of pre-images of $\eta$ under $f$ in $A$ is
\begin{equation} \label{eqn:count_on_tile}
N_\eta(f; A) = P(f; A) + \sum_{j=1}^n \wn{f \circ \gamma_j}{\eta}.
\end{equation}
\end{theorem}

\begin{proof}
We apply the argument principle to $f_\eta = f - \eta$ on $A$.
Note that $f_\eta$ is also non-degenerate, $J_f = J_{f_\eta}$, and 
$f_\eta$ has the same poles with same index as $f$,
so that $P(f_\eta; A) = P(f; A)$.
Since $f_\eta$ is non-zero on $\partial A$, it has no zeros in $\cM \cap
\overline{A}$.
Moreover, $f_\eta$ has only finitely many zeros in $A$.  For a bounded $A$ 
this holds since non-singular zeros are 
isolated~\cite[p.~413]{DurenHengartnerLaugesen1996}.
For $A_\infty$, assume that $f_\eta$ has infinitely many zeros in $A_\infty$ 
and hence in some $\{ z \in \C : \abs{z} \geq R \}$.
Then $f_\eta(1/z)$ has infinitely many non-singular zeros in 
$\{ z \in \C : \abs{z} \leq 1/R \}$, which contradicts the fact that 
such zeros are isolated.

First, suppose that $\Gamma$ is non-empty and that
$\gamma_1, \ldots, \gamma_n$ are closed Jordan curves.
If $f_\eta$ is sense-preserving in $A$, then $A$ lies to the left of 
$\gamma_1, \ldots, \gamma_n$.  
The argument principle implies
\begin{equation*}
\sum_{j=1}^{n} W(f_\eta; \gamma_j) = N_0(f_\eta; A) + \sum_{j=1}^k \ind(f_\eta; 
z_j)
= N_\eta(f; A) - P(f; A),
\end{equation*}
where we used that $f_\eta$ is sense-preserving and hence the index at a zero 
is $+1$ by~\eqref{eqn:index_nonsingular_zero} and negative at a pole by
Proposition~\ref{prop:poincare_index_poles}.
We obtain~\eqref{eqn:count_on_tile} in this case with $W(f_\eta; \gamma_j) = 
\wn{f \circ \gamma_j}{\eta}$; see~\eqref{eqn:winding_winding_number}.
Recall that $\wn{f \circ \gamma_j}{\eta} = 0$ if $f \circ \gamma_j$ is 
constant.
If $f_\eta$ is sense-reversing in $A$, then $A$ lies to the right 
of $\gamma_1, \ldots, \gamma_n$, and the index of $f$ at a zero is $-1$ 
by~\eqref{eqn:index_nonsingular_zero} and positive at a pole by 
Proposition~\ref{prop:poincare_index_poles}, and hence
\begin{equation*}
\sum_{j=1}^n W(f_\eta; - \gamma_j)
= -N_\eta(f; A) + P(f; A),
\end{equation*}
where $-\gamma_j$ denote the reversed curves.  Since $W(f; -\gamma_j) = - W(f; 
\gamma_j)$, we obtain~\eqref{eqn:count_on_tile}.

If some $\gamma_j$ is not a Jordan curve, then it self-intersects at a zero 
$z_0$ of $\omega'$, as indicated in 
Figure~\ref{fig:selfintersection_of_critical_set} (left).
However, $f_\eta$ is continuous and non-zero at $z_0$.
Hence, by an arbitrary small manipulation of $\gamma_j$, we obtain a Jordan 
curve on which $f_\eta$ has the same winding.
This is illustrated in Figure~\ref{fig:selfintersection_of_critical_set} 
(right).
The proof then remains unchanged with the new curves.

Finally, if $\Gamma$ is empty, then $A$ is the only component 
in $\cA$ and $N_\eta(f) = P(f)$ follows from 
Theorem~\ref{thm:sum_of_all_indices}.
\end{proof}

Summing over all $A \in \cA$ gives the total number of pre-images.

\begin{theorem}\label{thm:counting}
Let $f$ be a non-degenerate harmonic mapping.  Then $N_\eta(f)$, the number of 
pre-images in $\widehat{\C}$ of $\eta \in \C \setminus f(\cC)$ under $f$, is
\begin{equation*}
N_\eta(f) = P(f) + 2\sum_{\gamma \in \crit} \wn{f \circ \gamma}{\eta}.
\end{equation*}
Here $P(f) = \sum_{A \in \cA} P(f; A)$ denotes the number of poles of $f$ in 
$\widehat{\C}$ counted with the absolute 
values of their Poincar\'e indices, as in Theorem~\ref{thm:counting_on_tile}.
\end{theorem}

\begin{proof}
The function $f_\eta = f - \eta$ has no zeros on $\cC$, since $\eta$ is not 
a caustic point.  Let $\cA = \{A_1,\dots,A_m\}$ and denote by 
$\gamma_{1,j},\ldots, \gamma_{n_j,j}$ a parametrization of $(\partial A_j) 
\setminus \cM$ as above.  Applying 
Theorem~\ref{thm:counting_on_tile} for $A_1,\dots,A_m$ yields
\begin{equation*}
\begin{split}
N_\eta(f) &= \sum_{j = 1}^m N_\eta(f;A_j)
= \sum_{j = 1}^m \bigg(P(f;A_j) +  \sum_{k = 1}^{n_j} 
\wn{f\circ\gamma_{k,j}}{\eta} \bigg) \\
&= P(f) + 2\sum_{\gamma \in \crit} \wn{f \circ \gamma}{\eta}.
\end{split}
\end{equation*}
Here we used that every $\gamma_{k,j}$ consists of arcs which are boundary arcs 
of exactly two components in $\cA$, and that the critical curves are a 
(disjoint) parametrization of $\cC\setminus\cM$ according 
to~\eqref{eqn:parametrization}.
\end{proof}

\begin{remark} \label{rem:counting_formula}
Theorems~\ref{thm:counting_on_tile} and~\ref{thm:counting} not only contain a 
formula for counting the pre-images of $\eta$, but also allow to determine how 
the number of pre-images changes if $\eta$ changes its position relative to the 
caustics of $f$.
More precisely, the number of pre-images in $A \in \cA$ changes by $\pm 1$ if 
$\eta$ ``crosses'' a single caustic arc from $f(\partial A)$; see 
Theorem~\ref{thm:counting_on_tile}.
\end{remark}

For large enough $\abs{\eta}$, the pre-images are near the poles. This 
generalizes~\cite[Thm.~3.1]{LiesenZur2018a}.
We write $D_\eps(z_0) = \{ z \in \C : \abs{z-z_0} < \eps \}$.

\begin{theorem}\label{thm:large_eta}
Let $f$ be a non-degenerate harmonic mapping with poles 
$z_1,\dots,z_n$, let $\eps > 0$ be such that the sets $D_\infty = \{z \in 
\C : 
\abs{z} > \eps^{-1}\}$ and $D_{\eps}(z_1), \dots, D_{\eps}(z_n)$ are 
disjoint, and such that on each set $f$ is either sense-preserving or 
sense-reversing.
Then, for every $\eta \in \C$ with $\abs{\eta}$ large enough, we have
\begin{equation*}
N_\eta(f; D_{\eps}(z_k)) = \abs{\ind(f;z_k)} 
\quad \text{and} \quad 
N_\eta(f; D_\infty) = \abs{\ind(f - \eta; \infty)}.
\end{equation*}
Moreover, all pre-images of $\eta$ are in $D = \cup_{k=1}^n D_{\eps}(z_k) \cup 
D_\infty$.
\end{theorem}

\begin{proof}
Let $\eta \in \C$ be such that 
$\abs{f(z)} < \abs{\eta}$ for $z \in \partial D$,
which is possible since $\partial D$ is compact and $f$ continuous.
To apply Rouch\'e's theorem (e.g.~\cite[Thm.~2.3]{SeteLuceLiesen2015a}) to 
$f_\eta = f - \eta$ and $g(z) = - \eta$, note 
that
\begin{equation*}
\abs{f_\eta(z) - g(z)} = \abs{f(z)} < \abs{\eta}
\quad \text{for } z \in \partial D.
\end{equation*}
Since $f$ is either sense-preserving or sense-reversing on $D_{\eps}(z_k)$, 
we have
\begin{equation*}
0 = W(g; \gamma_k) = W(f_\eta; \gamma_k) = \pm 
N_\eta(f; D_\eps(z_k)) + \ind(f;z_k),
\end{equation*}
with $\gamma_k : \cc{0, 2\pi} \to \C$, $\gamma_k(t) = z_k + \eps e^{it}$.
Hence, $N_\eta(f; D_{\eps}(z_k)) = \abs{\ind(f;z_k)}$ as in 
Theorem~\ref{thm:counting_on_tile}.
Similarly, let $\gamma_\infty : \cc{0, 2\pi} \to \C$, 
$\gamma_\infty(t) = \eps^{-1} e^{-it}$, then
\begin{equation*}
0 = W(g; \gamma_\infty) = W(f_\eta; \gamma_\infty)
= \pm N_\eta(f; D_\infty) + \ind(f_\eta; \infty).
\end{equation*}

By increasing $\abs{\eta}$, so that $\eta$ lies outside all 
caustics, i.e., $\wn{f \circ \gamma}{\eta} = 0$ for all $\gamma 
\in \crit$, we have with Theorem~\ref{thm:counting}
\begin{equation*}
N_\eta(f) = P(f)
= \sum_{j=1}^n N_\eta(f; D_{\eps_j}(z_j)) + N_\eta(f;D_\infty).
\end{equation*}
This implies that all pre-images of $\eta$ are in $D$.
\end{proof}

Note that the number of pre-images determined in Theorem~\ref{thm:large_eta} is 
not necessarily the minimal number of pre-images as $\eta$ ranges over $\C 
\setminus f(\cC)$; see Example~\ref{ex:log} and Figure~\ref{fig:mpw_log}.
For non-singular harmonic polynomials, however, this is the lower bound for the 
number of zeros; see the discussion at the beginning of 
Section~\ref{sect:polynomials}.

We now consider $\eta$ as \emph{variable parameter}, and deduce the number of
pre-images of $\eta_2$ from the number of pre-images of another point 
$\eta_1$, e.g., with sufficiently large $\abs{\eta_1}$ as in 
Theorem~\ref{thm:large_eta}.

The caustics induce a partition of $\C \setminus f(\cC)$ into open and 
connected components, which we call \emph{caustic tiles}. 
This partition does not coincide with $f(\cA)$ in general, since
$f$ has not the open mapping property; see also Figure~\ref{fig:mpw_log}, where 
$\widehat{\C} \setminus \cC$ and $\C \setminus f(\cC)$ have a different number 
of (connected) components.
The winding number 
of $f\circ\gamma$ about $\eta$ depends on the position of $\eta$ with respect 
to the caustics, i.e., to which caustic tile $\eta$ belongs to.
The next theorem is an immediate and very useful consequence of 
Theorem~\ref{thm:counting}.

\begin{theorem} \label{thm:relative_counting}
For a non-degenerate harmonic mapping $f$ and non-caustic points 
$\eta_1, \eta_2 \in \C \setminus f(\cC)$, we have
\begin{equation} \label{eqn:relative_counting}
N_{\eta_2}(f) = N_{\eta_1}(f) + 2 \sum_{\gamma \in \crit} \big( \wn{f \circ 
\gamma}{\eta_2} - \wn{f \circ \gamma}{\eta_1} \big),
\end{equation}
and in particular:
\begin{enumerate}
\item If $\eta_1$ and $\eta_2$ are in the same caustic tile, then the number 
of pre-images under $f$ is the same, i.e., $N_{\eta_2}(f) = N_{\eta_1}(f)$.

\item If $\eta_1$ and $\eta_2$ are separated by a single caustic $f \circ 
\gamma$, then the number of pre-images under $f$ changes by two, i.e., 
$N_{\eta_2}(f) = N_{\eta_1}(f) \pm 2$.

\item $N_{\eta_1}(f)$ is odd if, and only if, $N_{\eta_2}(f)$ is odd.

\item Let $\eta_1, \eta_2 \in \C$.  If $N_{\eta_1}(f)$ is even and 
$N_{\eta_2}(f)$ is odd, then $\eta_1$ or $\eta_2$ is a caustic point of $f$.
\end{enumerate} 
\end{theorem}

We obtain a formula similar to~\eqref{eqn:relative_counting} for each set $A 
\in \cA$, using Theorem~\ref{thm:counting_on_tile} instead of 
Theorem~\ref{thm:counting}.  This yields $N_{\eta_2}(f; A) = N_{\eta_1}(f; 
A)$ in~1.  In~2., the number of pre-images increases/decreases by $1$ in 
the sets $A$ adjacent to the critical arc $\gamma$, and stays the same in all 
other sets $A$.

Items 3 and 4 are in the spirit of the ``odd number of images theorem'' 
from the theory of gravitational lensing in 
astrophysics~\cite[Thm.~11.5]{PettersLevineWambsganss2001}.

\subsection{Counting pre-images geometrically}

We determine geometrically whether the number of pre-images increases or 
decreases in item 2 of Theorem~\ref{thm:relative_counting}.
The key ingredient is the curvature of the caustics 
(Lemma~\ref{lem:tangent_to_caustic}), which allows to spot their orientation in 
a plot; see Figure~\ref{fig:caustic_index}.
Then, the change of the winding number $\wn{f \circ \gamma}{\eta_2} - \wn{f 
\circ \gamma}{\eta_1}$ can be determined with the next result.

\begin{proposition}[{\cite[Prop.~3.4.4]{Roe2015}}] 
\label{prop:winding_intersection}
Let $\gamma$ be a smooth closed curve and $\eta \notin \tr(\gamma)$.  
Let further $R$ be a ray from $\eta$ to $\infty$ in direction $e^{i\varphi}$, 
such 
	that $R$ is not a tangent at any point on $\gamma$.  Then $R$ intersects 
	$\gamma$ at finitely many points $\gamma(t_1),\dots,\gamma(t_k)$ and we have 
	for the winding number of $\gamma$ about $\eta$
	\begin{equation*}
	\wn{\gamma}{\eta} = \sum_{j=1}^k i_{t_j}(\gamma;R),
	\end{equation*} 
	where the \emph{intersection index} $i_{t_j}$ of $\gamma$ and $R$ at 
	$\gamma(t_j)$, is defined by 
	\begin{equation*}
	i_{t_j}(\gamma;R) =
	\begin{cases}
	+1, \quad \text{if } \im(e^{-i\varphi} \gamma'(t_j)) > 0, \\
	-1, \quad \text{if } \im(e^{-i\varphi} \gamma'(t_j)) < 0.
	\end{cases}
\end{equation*}
Recall that $e^{i \varphi}$ and $\gamma'(t_j)$ form a right-handed basis 
if $\im(e^{- i \varphi} \gamma'(t_j)) > 0$, and a left-handed basis if 
the imaginary part is negative.
\end{proposition}

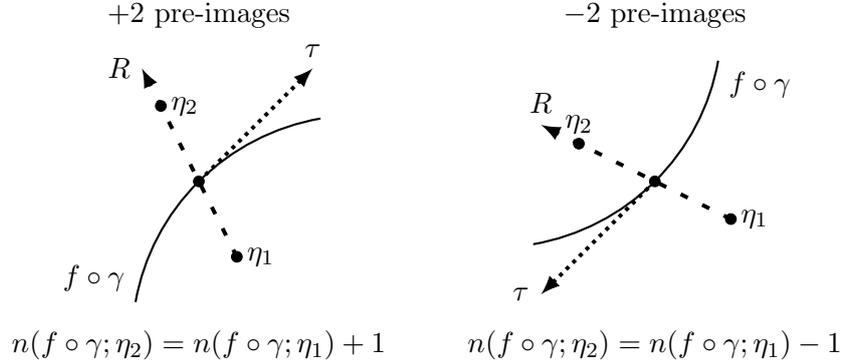
\begin{figure}[t]
	\center
	\begin{tikzpicture}
	\usetikzlibrary{positioning,arrows}
	\tikzset{arrow/.style={-latex}}
	
	\begin{scope}[shift={(6,0)}]
	\node at (-3, 1.2) {$-2$ pre-images};
	\draw[fill=black](-3,-1)circle(2pt);
	\draw [thick] (-3,-1) arc (-45:-80:3);
	\draw [thick] (-3,-1) arc (-45:-10:3) node[below right] {$f\circ \gamma$};
	\draw [line width=1.5pt,arrow, dotted] (-3,-1) to (-4.5,-2.5) node[left] {$\tau$};
	\draw[fill=black](-2,-1.5)circle(2pt) node[right] {$\eta_1$};
	\draw[fill=black](-4,-.5)circle(2pt) node[above] {$\eta_2$};
	\draw [line width=1.5pt, arrow,loosely dashed] (-2,-1.5) to (-4.5,-.25) node[above] {$R$};
	\draw[color=black] (-3,-3.5) node[above] {$\wn{f\circ\gamma}{\eta_2} 
= \wn{f\circ\gamma}{\eta_1} - 1 $};
\end{scope}
    
    \begin{scope}[shift={(-6,0)}]
    \node at (3, 1.2) {$+2$ pre-images};
    \begin{scope}[shift={(0,-1)}]
	\draw[fill=black](3,0)circle(2pt);
	\draw [thick] (3,0) arc (135:170:3) node[above left] {$f\circ \gamma$};
	\draw [thick] (3,0) arc (135:100:3);
	\draw [line width=1.5pt,arrow, dotted] (3,0) to (4.5,1.5) node[above] {$\tau$};
	\draw[fill=black](3.5,-1)circle(2pt) node[right] {$\eta_1$};
	\draw[fill=black](2.5,1)circle(2pt) node[right] {$\eta_2$};
	\draw [line width=1.5pt, arrow,loosely dashed] (3.5,-1) to (2.25,1.5) 
node[left] {$R$};
\end{scope}
	\draw[color=black] (3,-3.5) node[above] {$\wn{f\circ\gamma}{\eta_2} 
= \wn{f\circ\gamma}{\eta_1} + 1 $};
	\end{scope}
\end{tikzpicture}
\vspace{2mm}

\caption{Intersection index and caustics in the $\eta$-plane.}
\label{fig:caustic_index}
\end{figure}

Let $\eta_1$, $\eta_2$ be in two adjacent caustic tiles separated by a 
single caustic arc.  We call two sets adjacent, if they share a common boundary 
arc.  Consider the ray 
$R$ from $\eta_1$ to $\infty$ through $\eta_2$, and let it intersect the 
caustic between $\eta_1$ and $\eta_2$ at a fold point $(f \circ \gamma)(t_0)$.
Although the caustics are only piecewise smooth, we can smooth the 
finitely many (see Lemma~\ref{lem:tangent_to_caustic}) cusps as 
in~\cite[p.~16]{KhavinsonLeeSaez2018} to obtain a smooth curve with same winding 
numbers about $\eta_1$ and $\eta_2$.
Then $\wn{f \circ \gamma}{\eta_1} 
= \wn{f \circ \gamma}{\eta_2} + i_{t_0}(f \circ \gamma; R)$ by 
Proposition~\ref{prop:winding_intersection}, and equivalently
\begin{equation*}
\wn{f \circ \gamma}{\eta_2} - \wn{f \circ \gamma}{\eta_1}
= - i_{t_0}(f \circ \gamma; R),
\end{equation*}
where the intersection index is $+1$ if $\eta_2 - \eta_1$ and $\tau(t_0)$ form 
a right-handed basis, and $-1$ if the two vectors form a left-handed basis; see 
Figure~\ref{fig:caustic_index}.

Caustic tiles have three different shapes.
We call a caustic tile $B$ \emph{deltoid-like} (respectively 
\emph{cardioid-like}), if for every point $z_0 \in \partial B$, for which the 
tangent to the caustics exists and is non-zero, there exists an open disk $D$ 
centered at $z_0$ such that the intersection of $D$ and the tangent line to 
$\partial B$ at $z_0$ is contained in $B$ (respectively contained in $\C 
\setminus B$).
We call a caustic tile \emph{mixed}, if it is neither deltoid nor 
cardioid-like.
In Figure~\ref{fig:mpw_log} (middle right), the tiles with the number $6$ are 
deltoid-like, the tile with the number $2$ is cardioid-like, and the tile with 
the number $4$ is a mixed caustic tile.
Entering a deltoid-like tile gives two additional pre-images,
entering a cardioid-like tile gives two fewer pre-images,
for a mixed tile both occur according to the shape of the ``crossed'' caustic 
arc; see Figure~\ref{fig:caustic_index} and Example~\ref{ex:log}.

\begin{figure}[th!]
{\centering
\hspace{0.002\linewidth}
\includegraphics[width=0.467\linewidth,height=0.47\linewidth]
{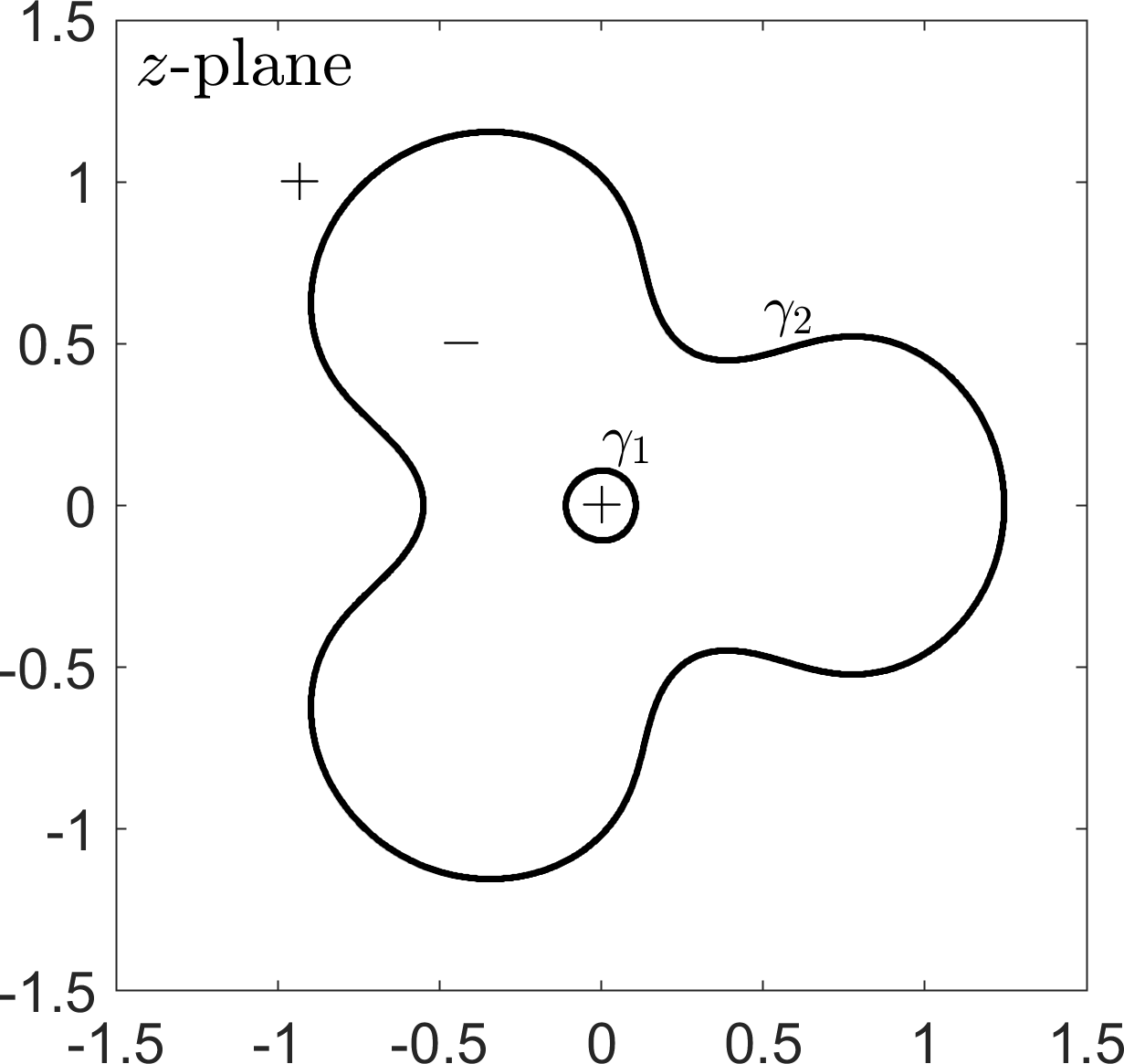}
\hspace{0.013\linewidth}
\includegraphics[width=0.468\linewidth,height=0.47\linewidth]
{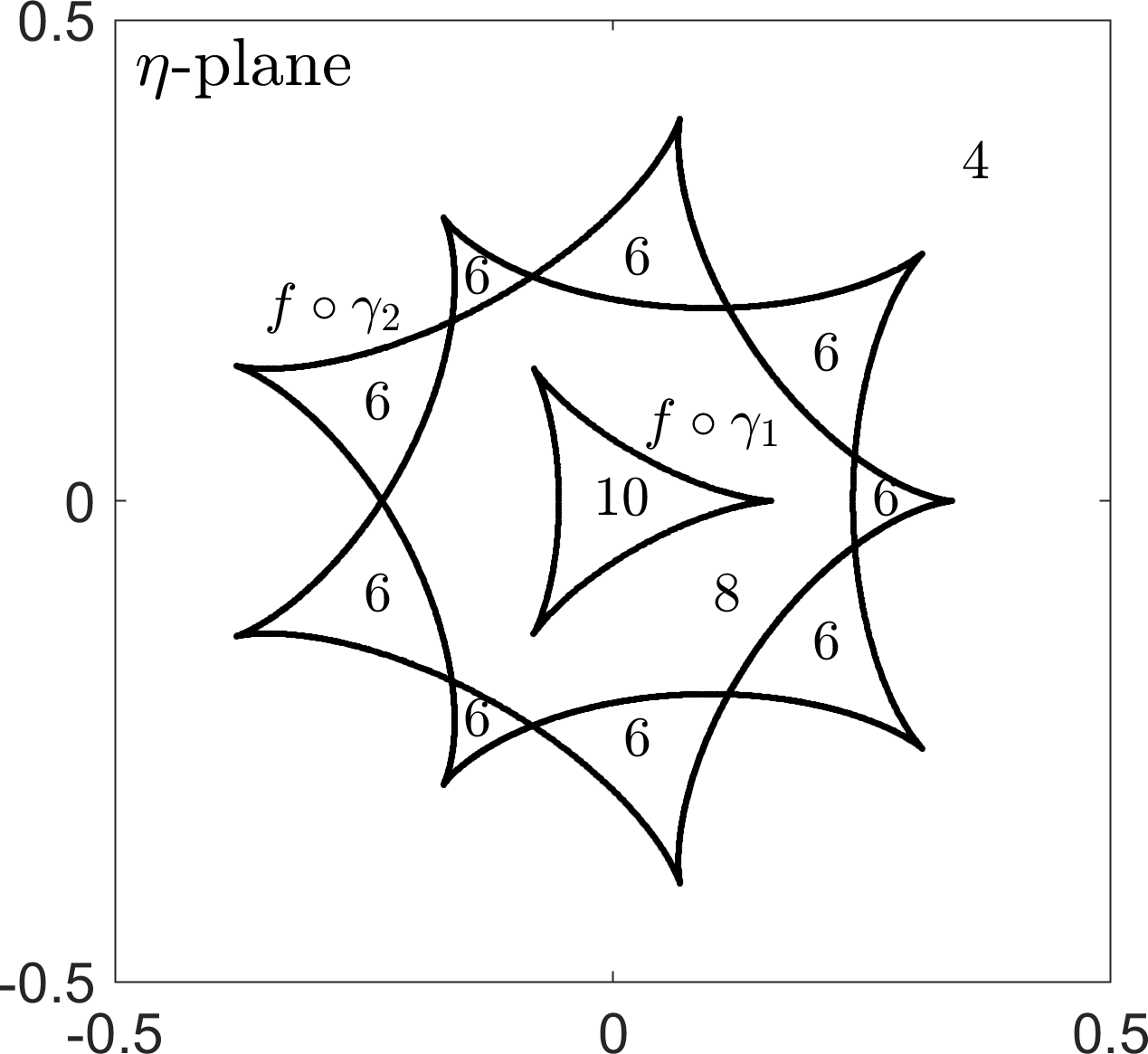}
\hspace{0.008\linewidth}
}

\vspace{2mm}

{\centering
\includegraphics[width=0.44\linewidth,height=0.44\linewidth]
{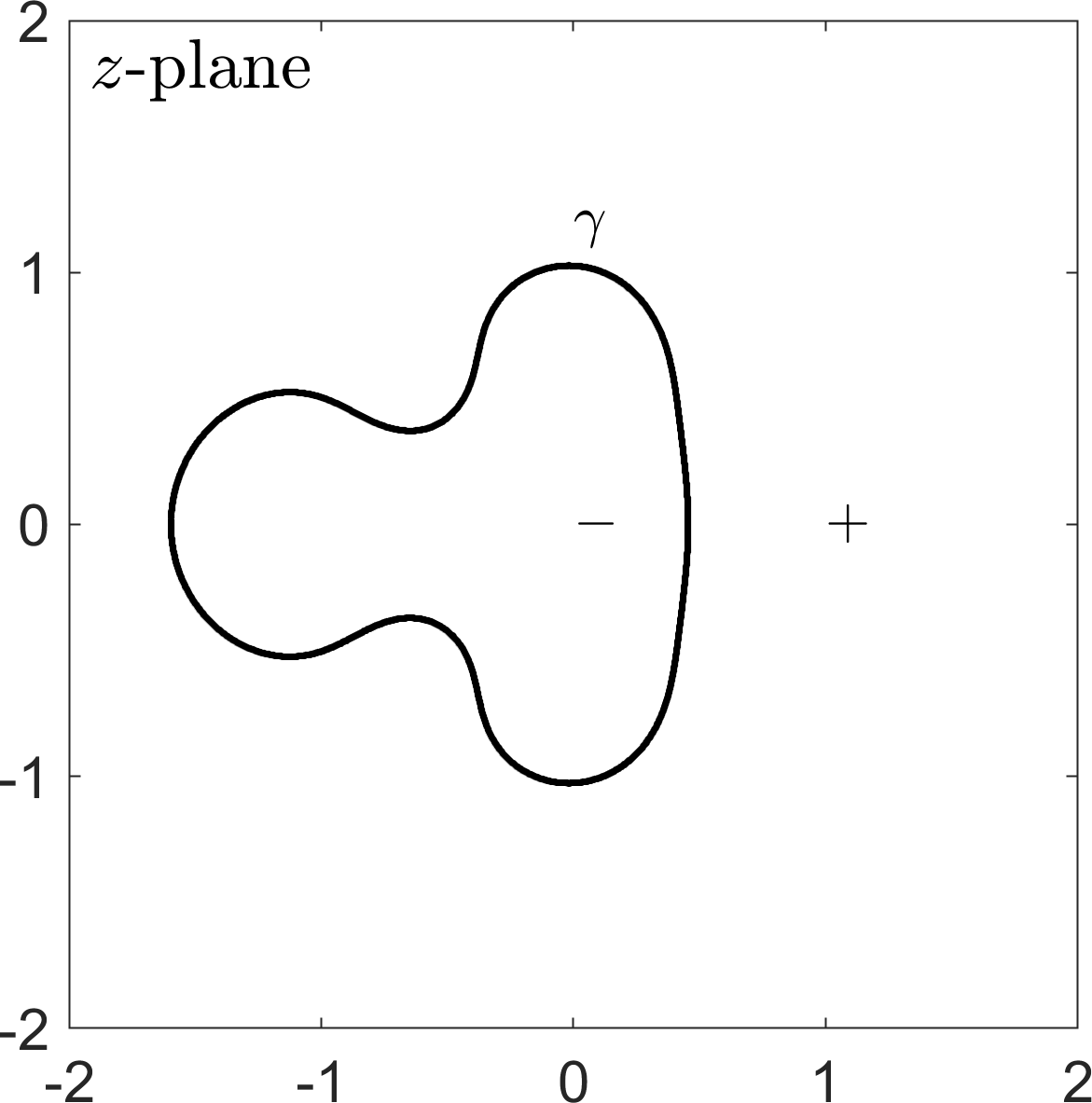}
\hspace{0.04\linewidth}
\includegraphics[width=0.44\linewidth,height=0.44\linewidth]
{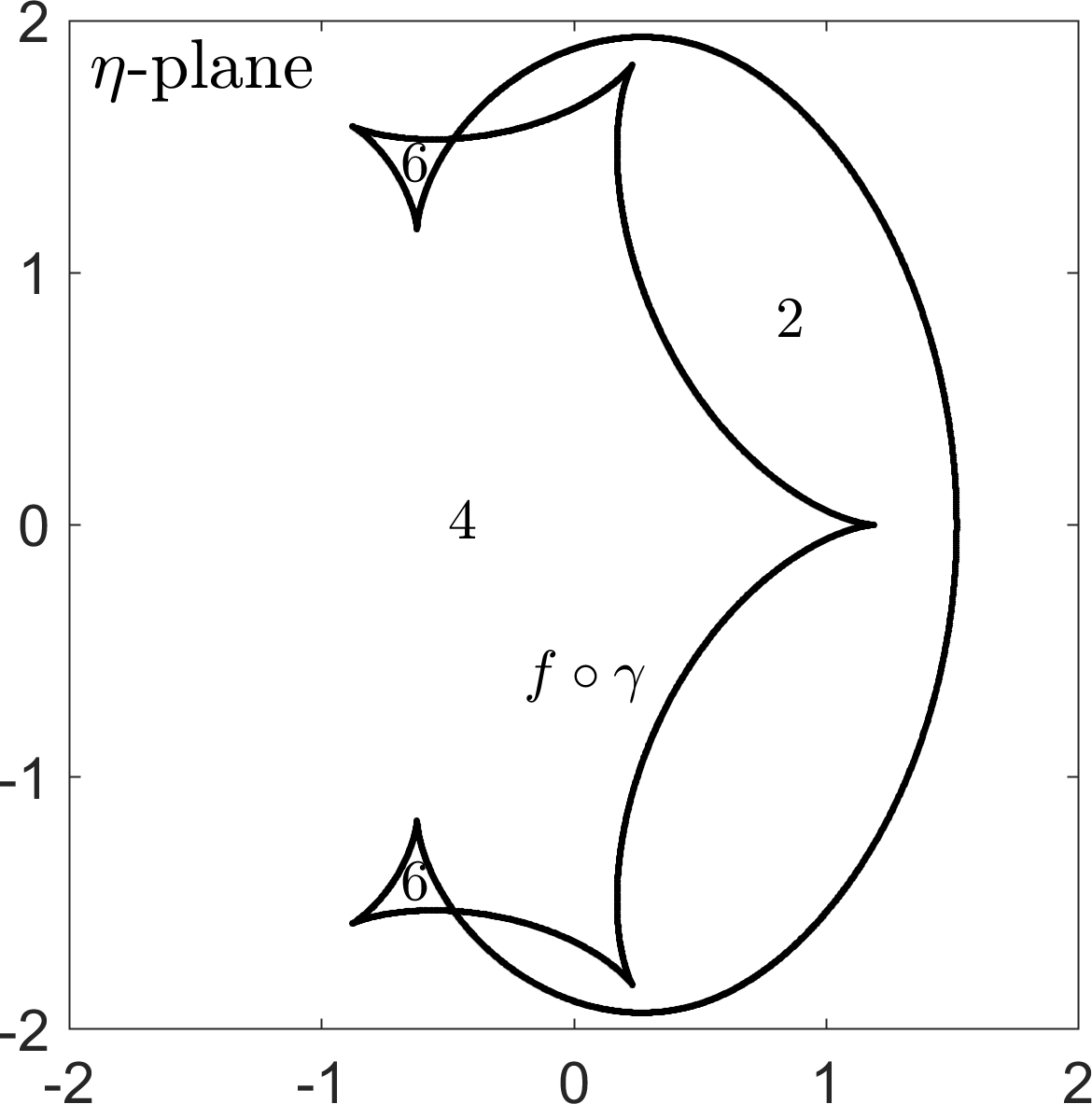}

}

\vspace{2mm}

{\centering
\includegraphics[width=0.44\linewidth,height=0.43\linewidth]
{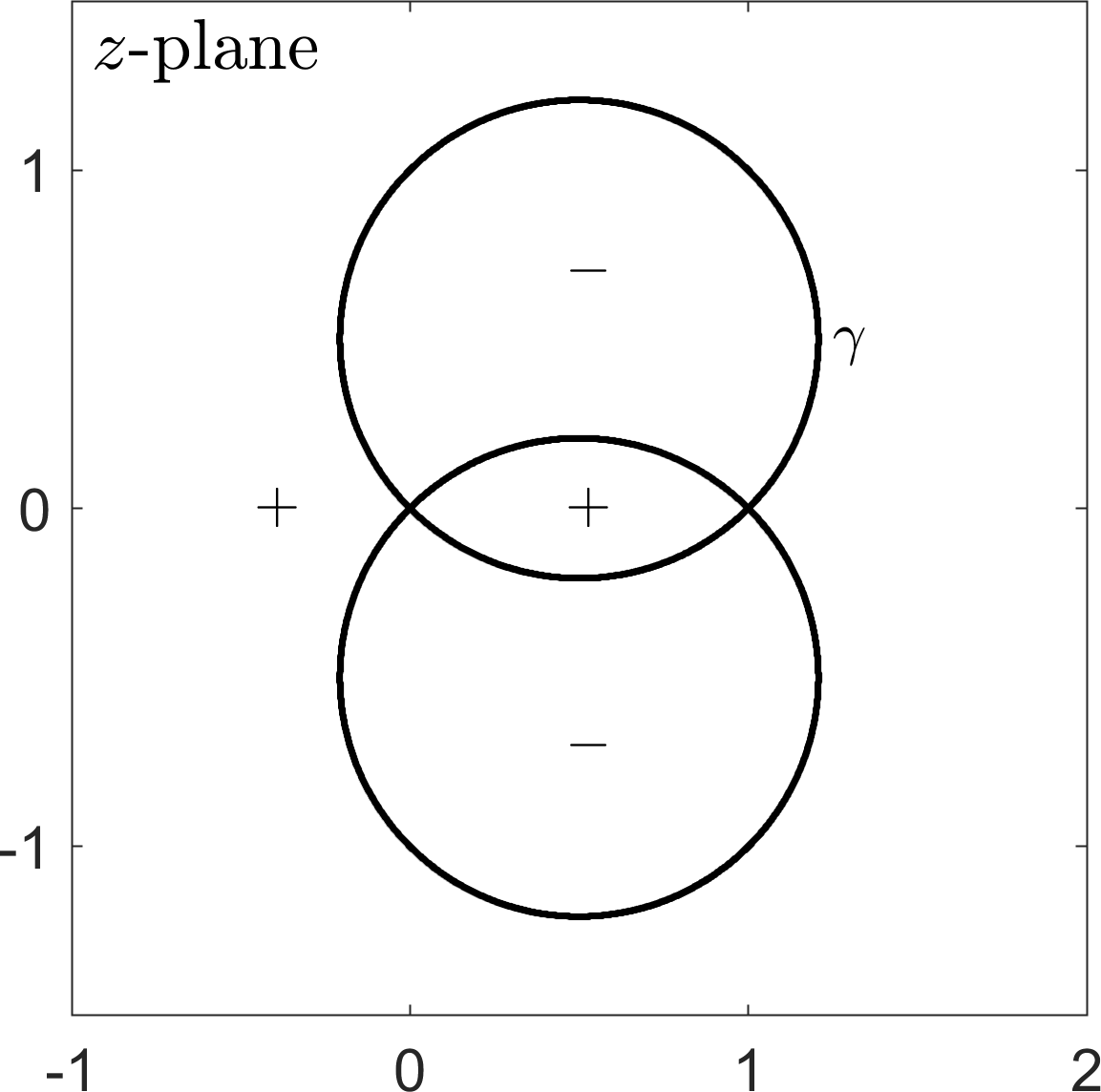}
\hspace{0.04\linewidth}
\includegraphics[width=0.44\linewidth,height=0.44\linewidth]
{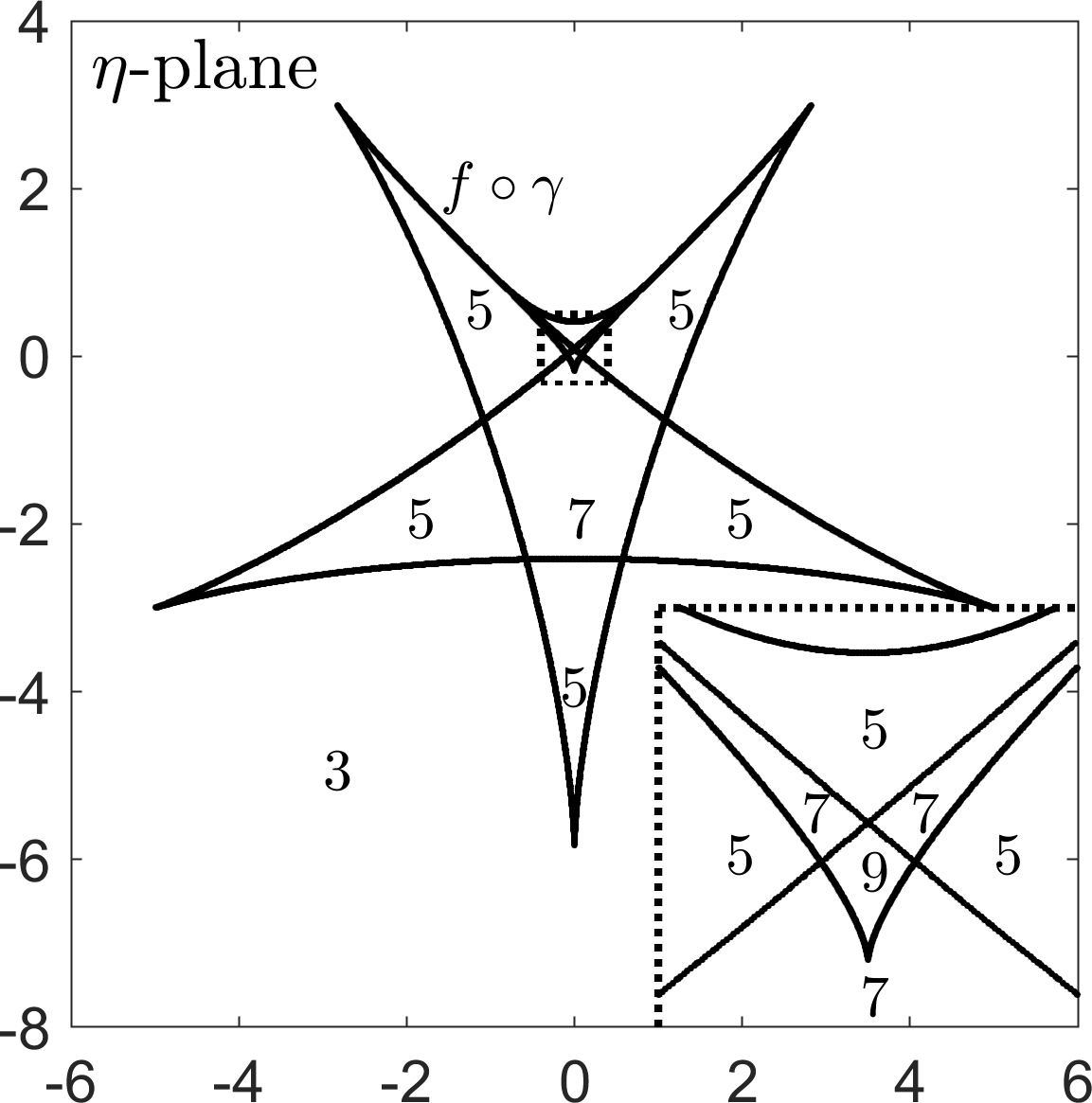}

\vspace{2mm}

}
\caption{Critical curves (left) and caustics (right) of the functions in 
Examples~\ref{ex:mpw} (top), \ref{ex:log} (middle), and~\ref{ex:wilmshurst} 
(bottom).  The $+/-$ 
signs indicate the regions where $f$ is sense-preserving/sense-reversing.
The numbers indicate the number of pre-images of an $\eta$ in the respective 
caustic tile.  The dotted line in the bottom right plot marks a zoom-in.}
\label{fig:mpw_log}
\end{figure}

\begin{example} \label{ex:mpw}
Consider the non-degenerate rational harmonic mapping
\begin{equation*}
f(z) = z - \conj{\left( \frac{z^2}{z^3 - 0.6^3} \right)}.
\end{equation*}
Figure~\ref{fig:mpw_log} (top) shows the critical set and the caustics of $f$.
We have $P(f) = 4$, since $f$ has four simple poles ($\infty$ with index $-1$, 
the others with index $1$); see 
Proposition~\ref{prop:poincare_index_poles}.  Thus, for 
$\eta$ in the outer region, i.e., with $\wn{f\circ\gamma_j}{\eta} = 0$ for 
$j=1,2$, we have $N_\eta(f) = 4 + 2 \cdot 0 = 4$.  For $\eta = 0$,
we have $\wn{f\circ\gamma_1}{0} = 1$ and $\wn{f\circ\gamma_2}{0} = 2$, 
so that $f$ has $N_0(f) = 4 + 2 \cdot 3 = 10$ zeros.
\end{example}

Certain rational harmonic mappings are studied in 
gravitational lensing in astrophysics; 
see~e.g.~\cite{KhavinsonNeumann2008,LuceSeteLiesen2014a}.
Also transcendental functions such as $f(z) = z - \conj{k/\sin(z)}$ appear in 
this context~\cite{BergweilerEremenko2010}.

\begin{example}\label{ex:log}
Figure~\ref{fig:mpw_log} (middle) shows the critical curves and caustics of the 
non-degenerate harmonic mapping
\begin{equation*}
f(z) = z^2 + \frac{1}{\conj z} + \frac{1}{\conj z+1} + 2\log\abs{z}.
\end{equation*}
Here, $P(f) = 4$ from the simple poles at $0$ and $-1$ with index $+1$ and 
the double pole at $\infty$ with index $-2$; see 
Proposition~\ref{prop:poincare_index_poles}.
Consequently, any $\eta$ in the outer region (i.e., with 
$\wn{f\circ\gamma}{\eta} = 0$) has $4$ pre-images.
Note the effect of deltoid-like, cardioid-like and mixed caustic tiles 
described above: the tiles where $\eta$ has $6$ pre-images are deltoid-like,
the tile where $\eta$ has $2$ pre-images is cardioid-like, and the outer tile 
is mixed.
\end{example}

\begin{example} \label{ex:wilmshurst}
The non-degenerate harmonic polynomial
\begin{equation*}
f(z) = p(z) + \conj{q(z)} = z^n + (z-1)^n + \conj{i z^n - i (z-1)^n}, \quad n 
\geq 1,
\end{equation*}
has the maximum number of $n^2$ zeros~\cite[p.~2080]{Wilmshurst1998}.
Its critical set consists of $n-1$ circles, intersecting in $0$ and $1$, and 
can be parametrized as discussed in Section~\ref{sect:number_of_preimages}.
Figure~\ref{fig:mpw_log} (bottom) shows the critical set and caustics for $n = 
3$.
\end{example}

\section{Location of pre-images near the critical set}
\label{sect:local_valence}

In Section~\ref{sect:global_valence}, we omitted the case when $\eta$ is on a 
caustic.
Here, we study the local effect when $\eta$ ``crosses'' a caustic, i.e., when 
the number of pre-images changes.  Since this is a local effect, the harmonic 
mappings are neither required to be globally defined nor to be non-degenerate.

Non-singular pre-images persist under a small change of $\eta$, which is an 
immediate consequence of the inverse function theorem.

\begin{proposition} \label{prop:local_change}
Let $f$ be a harmonic mapping defined in the open set $\Omega \subseteq \C$ 
and let $f$ be non-singular at $z_0 \in \Omega$.
Then there exist open neighborhoods $U \subseteq \Omega \setminus \cC$ of 
$z_0$ and $V$ of $f(z_0)$ such that each $\eta \in V$ has exactly 
one pre-image under $f$ in $U$.
\end{proposition}

Lyzzaik~\cite{Lyzzaik1992} investigated the local behavior of
light harmonic mappings,
defined on an open and simply connected subset of $\C$.
His analysis relies upon the local transformation of $f$ near a critical point 
$z_0 \in \cC$ into standard mappings $h_2\circ f\circ h_1^{-1}(z) = z^n$ or 
$h_2\circ f\circ h_1^{-1}(z) = \conj z^n$, where $h_1$ and $h_2$ are 
sense-preserving homeomorphisms; see~\cite[Sect.~3]{Lyzzaik1992} for details. 
If such a standard mapping exists we write $f_{z_0} \sim z^n$ and $f_{z_0} \sim 
\conj z^n$ respectively.
One of Lyzzaik's results is the following: Let $f(z_0)$ be a fold and $U$ be a
neighborhood of $z_0$.  Then there exists a partition 
$U_1, U_2$ of $U \setminus \cC$ with $f_{z_0} \sim z$ in~$\overline{U}_1$ and 
$f_{z_0} \sim \conj z$ in~$\overline{U}_2$. Similarly, if $f(z_0)$ is a cusp 
and $h'(z_0) \neq 0$, we have $f_{z_0} \sim z^3$ in~$\overline{U}_1$, $f_{z_0} 
\sim \conj z$ in~$\overline{U}_2$ or $f_{z_0} \sim z$ in~$\overline{U}_1$, 
$f_{z_0} \sim \conj z^3$ in $\overline{U}_2$; see~\cite[Thm.~5.1]{Lyzzaik1992}.
This allows to determine the \emph{valence}
\begin{equation*}
V(f;U) = \sup_{\eta \in \C} \, N_\eta(f; U)
= \sup_{\eta \in \C} \, \abs{ \{ z \in U : f(z) = \eta \} }
\end{equation*}
of $f$ in $U$.
In particular we have
\begin{equation} \label{eqn:valence}
V(f; D_\eps(z_0)) = \begin{cases} 2, &\text{if } f(z_0) \text{ is a fold}, \\
3, &\text{if } f(z_0) \text{ is a cusp with } h'(z_0) \neq 0, 
\end{cases}
\end{equation}
for sufficiently small $\eps > 0$; see~\cite[Thm~5.1]{Lyzzaik1992}.
However, the above transformations are not immediately available for practical 
computations in general.

We complement Lyzzaik's work by investigating which values near a fold $\eta = 
f(z_0)$ have actually $2$, $1$ or no pre-images under $f$ in $D_\eps(z_0)$, and 
by approximately locating the pre-images for certain $\eta$. For this we use 
convergence results on the harmonic Newton iteration
\begin{equation}\label{eqn:harmonic_newton}
z_{k+1} = z_k - \frac{\conj{h'(z_k)} f(z_k) - \conj{g'(z_k)f(z_k)}}{J_f(z_k)}, 
\quad k \geq 0,
\end{equation}
from~\cite{SeteZur2020}.  If the sequence~\eqref{eqn:harmonic_newton} converges 
and all iterates $z_k$ are in $D \subseteq \C$, then there 
exists a zero of $f$ in $\overline{D}$.
The proof of the next theorem relies on this strategy.

\begin{theorem}\label{thm:local_fold}
Let $f$ be a light harmonic mapping and $z_0 \in \cC \setminus \cM$, such that 
$\eta = f(z_0)$ is a fold.  Moreover, let
\begin{equation*}
f(z) = \sum_{k=0}^\infty a_k (z-z_0)^k + \conj{\sum_{k=0}^\infty b_k (z-z_0)^k}
\quad \text{and} \quad
c = - \left( \frac{a_2 \conj{b}_1}{a_1} + \frac{\conj{b}_2 a_1}{\conj{b}_1} 
\right).
\end{equation*}
Then, for all sufficiently small $\eps > 0$, there exists a $\delta > 0$, 
such that for all $0 < t < \delta$ we have:
\begin{enumerate}
\item $\eta + t c$ has exactly two pre-images under $f$ in $D_\eps(z_0)$,
\item $\eta$ has exactly one pre-image under $f$ in $D_\eps(z_0)$,
\item $\eta - t c$ has no pre-image under $f$ in $D_\eps(z_0)$.
\end{enumerate}
In case 1, each disk $\{z \in \C : \abs{z - z_\pm} \leq \text{const} 
\cdot t \}$, where $z_\pm = z_0 \pm i \sqrt{t \, \conj{b}_1/a_1}$, contains one 
of the two pre-images, and $f$ is sense-preserving at one and sense-reversing 
at the other.
\end{theorem}

\begin{proof}
Since $z_0 \in \cC$ and $f(z_0)$ is a fold, we have $h'(z_0) \neq 0$, and hence
$\abs{g'(z_0)} = \abs{h'(z_0)} \neq 0$.  Then there exists $\theta \in 
\co{0, \pi}$ with $\conj{b}_1 = a_1 e^{i2\theta}$, and
\begin{equation*}
c
= - a_1 e^{i \theta} \left( \frac{a_2}{a_1} e^{i \theta} + 
\conj{\frac{b_2}{b_1} e^{i \theta}} \right)
\end{equation*}
is non-zero by Lemma~\ref{lem:cusp_condition}.

1.
We apply the harmonic Newton iteration~\eqref{eqn:harmonic_newton} to the 
shifted function $f_{\eta+t c} = f - (\eta + t c)$
with initial points $z_\pm$.  
By~\cite[Lem.~5.1, Thm.~5.2]{SeteZur2020} and their proofs, the respective 
sequences of 
iterates remain in $D_\pm$,
and converge to two distinct zeros of $f_{\eta+t c}$ for all sufficiently 
small $t > 0$. 
Thus, $\eta+t c$ has exactly two pre-images under $f$ in $D_\eps(z_0)$, using~\eqref{eqn:valence}.

2.
Since $f$ is light and $f(z_0) = \eta$, there exists $\eps > 0$ such that $z_0$ 
is the only pre-image of $\eta$ in $D_\eps(z_0)$.

3.  We show first that the ``direction'' $c$ is not tangential to the 
caustic, and hence that $\eta + t c$ and 
$\eta-t c$ are not in the same caustic tile.
Since $\eta = f(z_0)$ is a fold, we have with $z_0 = \gamma(t_0)$ and the 
tangent $\tau$ from Lemma~\ref{lem:tangent_to_caustic}
\begin{equation*}
\conj{\tau(t_0)} c
= - \psi(t_0) e^{i t_0/2} a_1 e^{i \theta} \left( \frac{a_2}{a_1} e^{i \theta} 
+ \conj{\frac{b_2}{b_1} e^{i \theta}} \right)
= \mp \psi(t_0) \abs{a_1} \left( \frac{a_2}{a_1} e^{i \theta} + 
\conj{\frac{b_2}{b_1} e^{i \theta}} \right),
\end{equation*}
since $e^{i t_0/2} e^{i \theta} a_1 = \pm \abs{a_1}$; see the proof of
Lemma~\ref{lem:cusp_condition}.
Since $\psi$ is real, and non-zero 
at a fold, we have $\im(\conj{\tau(t_0)}c) \neq 0$ by 
Lemma~\ref{lem:cusp_condition}.  
Hence, for a sufficiently small $t > 0$, the points $\eta + t c$ and $\eta-t c$ 
are on different sides of the caustic $f \circ \gamma$,
where $\gamma$ denotes the critical curve through $z_0$.
Thus, 
there are either $2+2 = 4$ or $2-2=0$ pre-images of $\eta - t c$ under 
$f$ in $D_\eps(z_0)$; see Theorem~\ref{thm:counting} if $f$ is 
non-degenerate, and~\cite[Thm.~6.7]{Neumann2005} for light harmonic mappings.
Since $V(f; D_\eps(z_0)) = 2$ by~\eqref{eqn:valence}, only the latter case is 
possible.

Moreover, the two pre-images of $\eta + t c$ in 1.\@ lie on different 
sides of the corresponding critical arc, and hence $f$ is sense-preserving at 
one pre-image and sense-reversing at the other;
see Theorem~\ref{thm:counting_on_tile} and Remark~\ref{rem:counting_formula} if 
$f$ is non-degenerate, and again~\cite[Thm.~6.7]{Neumann2005} for light 
harmonic mappings.
\end{proof}

\begin{figure}[t]
	\center
	\begin{tikzpicture}
	\usetikzlibrary{positioning,arrows}
	\tikzset{arrow/.style={-latex}}
	\draw[fill=black](-3,-1)circle(2pt);
	\node at (-3, 1.6) {$z$-plane};
	\draw [thick] (-3,-1) arc (-45:-80:3);
	\draw [thick] (-3,-1) arc (-45:-10:3) node[below right] {$\gamma$};
	\draw [line width=1.5pt, arrow,loosely dashed] (-3,-1) to (-4.5,-.25) 
node[above] {$z_1$};
	\draw [line width=1.5pt, arrow,loosely dashed] (-3,-1) to (-1.5,-1.75) 
node[above] {$z_2$};
		
    \node at (3, 1.6) {$\eta$-plane};
	\draw[fill=black](3,0)circle(2pt);
	\draw [thick] (3,0) arc (135:170:3) node[above left] {$f\circ \gamma$};
	\draw [thick] (3,0) arc (135:100:3);
	\draw [line width=1.5pt, arrow,loosely dashed] (3,0) to (2.5,1) node[left] 
{$\eta$};
	\draw [line width=1.5pt, dotted] (3.5,-1) to (3,0);
	\end{tikzpicture}

\vspace{0.4cm}

\begin{tikzpicture}
	\usetikzlibrary{positioning,arrows}
	\tikzset{arrow/.style={-latex}}
	\draw[fill=black](-3,-1)circle(2pt);
	\draw [thick] (-3,-1) arc (-45:-80:3);
	\draw [thick] (-3,-1) arc (-45:-10:3) node[left] {$\gamma$};

	\draw [line width=1.5pt, arrow,loosely dashed] (-3,-1) to (-4.5,-.25) 
node[above] {$z_1$};
	\draw [line width=1.5pt, dotted] (-3,-1) to (-1.5,-1.75) node[below] 
{$z_1$};
	\draw [line width=1.5pt, arrow,loosely dashed] (-3,-1) to (-1.5,.5) 
node[right] {$z_3$};
	\draw [line width=1.5pt, arrow,loosely dashed] (-3,-1) to (-4.5,-2.5) 
node[left] {$z_2$};
		
    \begin{scope}[shift={(0,-0.5)}]
	\draw[fill=black](2,-.5)circle(2pt);
	\draw [thick] (2,-.5) arc (270:349:2);
	\draw [thick] (2,-.5) arc (90:11:2) node[left] {$f\circ \gamma$};
	\draw [line width=1.5pt, arrow,loosely dashed] (2,-.5) to (4,-.5) 
node[above] {$\eta$};
	\draw [line width=1.5pt, dotted] (0.7,-.5) to (2,-.5);
	\end{scope}
	\end{tikzpicture}

\caption{Behavior at a fold (top) and cusp (bottom); cf.~\cite[Figs.~4, 
7]{LiesenZur2018a}.}
\label{fig:caustic_crossing}
\end{figure}
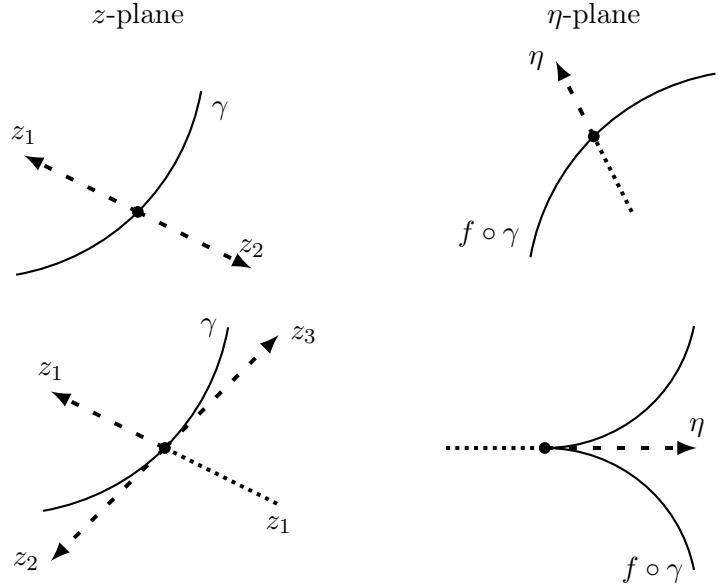

Figure~\ref{fig:caustic_crossing} (top) illustrates the effect in 
Theorem~\ref{thm:local_fold}.  The points $z_1$, $z_2$ are the pre-images of 
$\eta + t c$ under $f$, i.e., the limits of the harmonic Newton iteration for
$f - (\eta + t c)$ with initial points $z_\pm$.

\begin{remark}\label{rem:fold_crossing}
\begin{enumerate}
\item From the proof of Lemma~\ref{lem:cusp_condition} we have 
$\im(\conj{\tau(t_0)} c) > 0$, i.e., $\tau(t_0)$ and $c$ form a right-handed 
($\R$-)basis.
Combining Theorem~\ref{thm:local_fold} with Proposition~\ref{prop:local_change} 
allows to replace $c$ by any direction $d$ with $\im(\conj{\tau(t_0)} d) > 0$ 
without changing the number of pre-images in $D_\eps(z_0)$.
More generally, if $\widetilde{\eta}$ is in the same caustic tile as $\eta + 
t c$ (the tile containing the tangent) and close enough to $\eta$, then 
$\widetilde{\eta}$ has $2$ 
pre-images under $f$ in $D_\eps(z_0)$, and similarly in the other cases.

\item For a fold $\eta$ with several pre-images in $\cC$, the effect of 
Theorem~\ref{thm:local_fold} happens at all points in $f^{-1}(\{ \eta \}) \cap 
\cC$ simultaneously; see Example~\ref{exp:N}.

\item Theorem~\ref{thm:local_fold} only covers pre-images in $D_\eps(z_0)$.
All other non-singular pre-images of $\eta$ 
under $f$ persist by Proposition~\ref{prop:local_change}, when going from 
$\eta$ to 
$\eta \pm t c$, provided that $t > 0$ is 
sufficiently small.
\end{enumerate}
\end{remark}

When $\eta$ is a cusp as in~\eqref{eqn:valence}, we have a similar 
result, which is also based on the harmonic Newton iteration; 
see~\cite[Thm.~5.2, 2.]{SeteZur2020}.
For $\widetilde{\eta}$ close enough to $\eta$ on one side of the caustic, there 
are $3$ pre-images by~\cite[Thm~5.1]{Lyzzaik1992}, and on the other side 
there is only $1$ pre-image by Proposition~\ref{prop:local_change} and 
Theorem~\ref{thm:local_fold}; see Figure~\ref{fig:caustic_crossing} (bottom).

The next example illustrates the local behavior 
near critical points corresponding to a fold, a cusp, and a double fold, and 
near a point in $\cM$.

\begin{example}\label{exp:N}
We consider the harmonic mapping $f(z) = \frac{1}{3}z^3 + \frac{1}{2} 
\conj{z}^2$, which is similar to the one in~\cite[Ex.~5.17]{Neumann2005}.
Since $J_f(z) = 
\abs{z}^2 - \abs{z}^4$, we have $\cC = \partial \bD \cup \{0\}$ and $\cM = 
\{0\}$.  The caustics of $f$ are shown in Figure~\ref{fig:poly_caus}, together 
with certain points $\eta_1, \ldots, \eta_6$.
While ``moving'' $\eta$ from $\eta_1 = -0.4$ to $\eta_6 = 0.9$ we reach a 
double fold, a point in $f(\cM)$, a fold and a cusp.
The respective pre-images of $\eta_j$ under $f$ are shown in 
Figure~\ref{fig:N_exp}, and have been computed with the harmonic Newton 
method~\cite{SeteZur2020}.
The background is colored according to the \emph{phase} 
$f_{\eta_j}/\abs{f_{\eta_j}}$ of the shifted function $f_{\eta_j} = f - 
\eta_j$; see~\cite{Wegert2012} for an 
extensive discussion of phase plots.
The Poincar\'e index of $f$ at $z_0$ corresponds to the color change on 
a small circle around $z_0$ in the positive direction.  In particular, we 
have $\ind(f_{\eta_j};z_0) = +1$ for zeros $z_0$ in $\C \setminus 
\overline{\bD}$, and $\ind(f_{\eta_j};z_0) = -1$ for zeros $z_0$ in $\bD 
\setminus \{ 0 \}$; see also Proposition~\ref{prop:poincare_index_zeros}.
A feature is the zero $0 \in \cM$ of $f$, for which $\ind(f; 0) = -2$ 
by~\eqref{eqn:index_zero_in_N}.
This reflects the fact that two pre-images where $f$ is sense-reversing merge 
together at $0$; see Remark~\ref{rem:zeros_in_M}.
\end{example}

\begin{figure}[t!]
	{\centering
		\includegraphics[width=0.45\linewidth, height = 
0.43\linewidth]{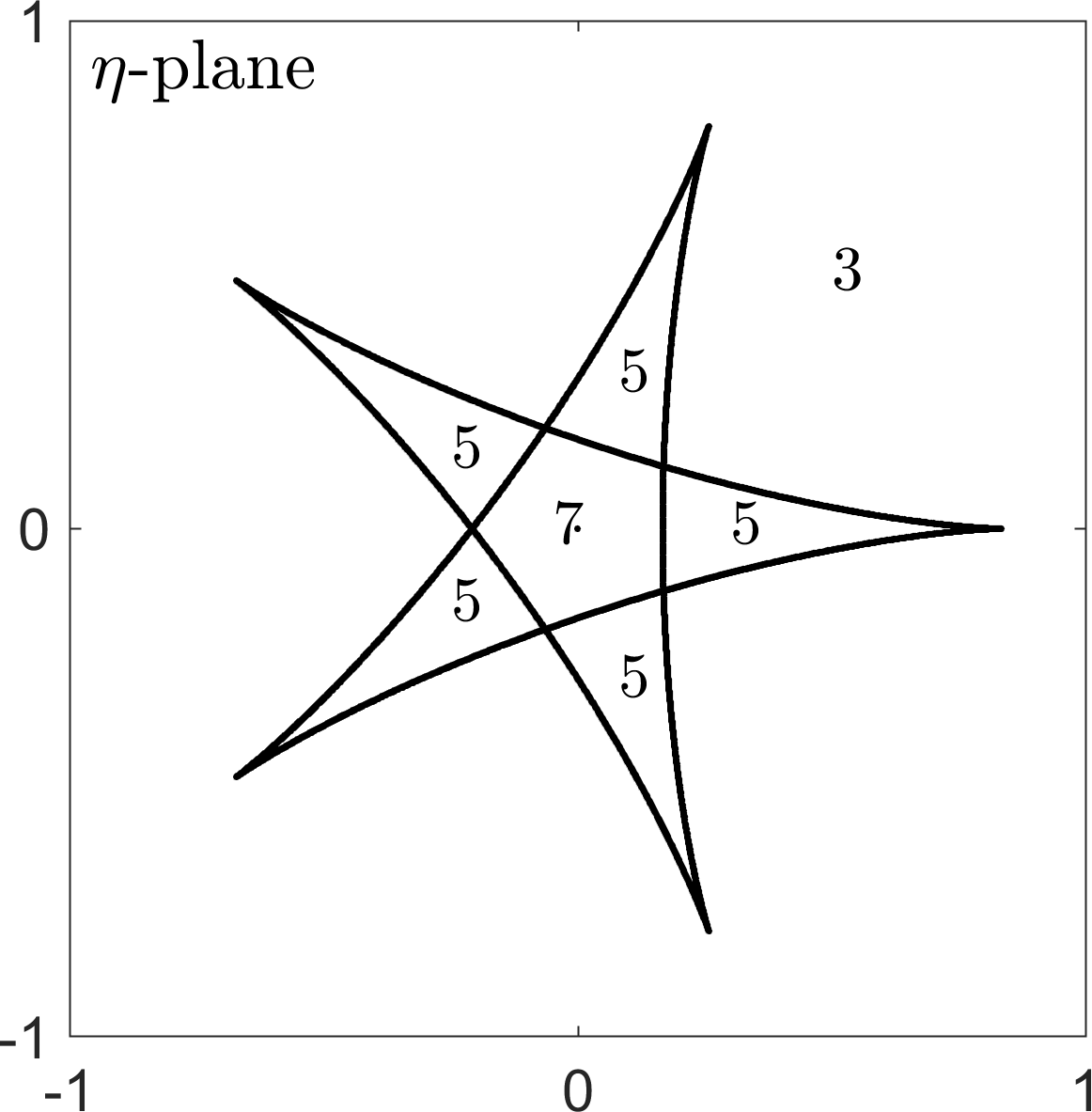}
\hspace{0.02\linewidth}
		\includegraphics[width=0.45\linewidth, height = 
0.43\linewidth]{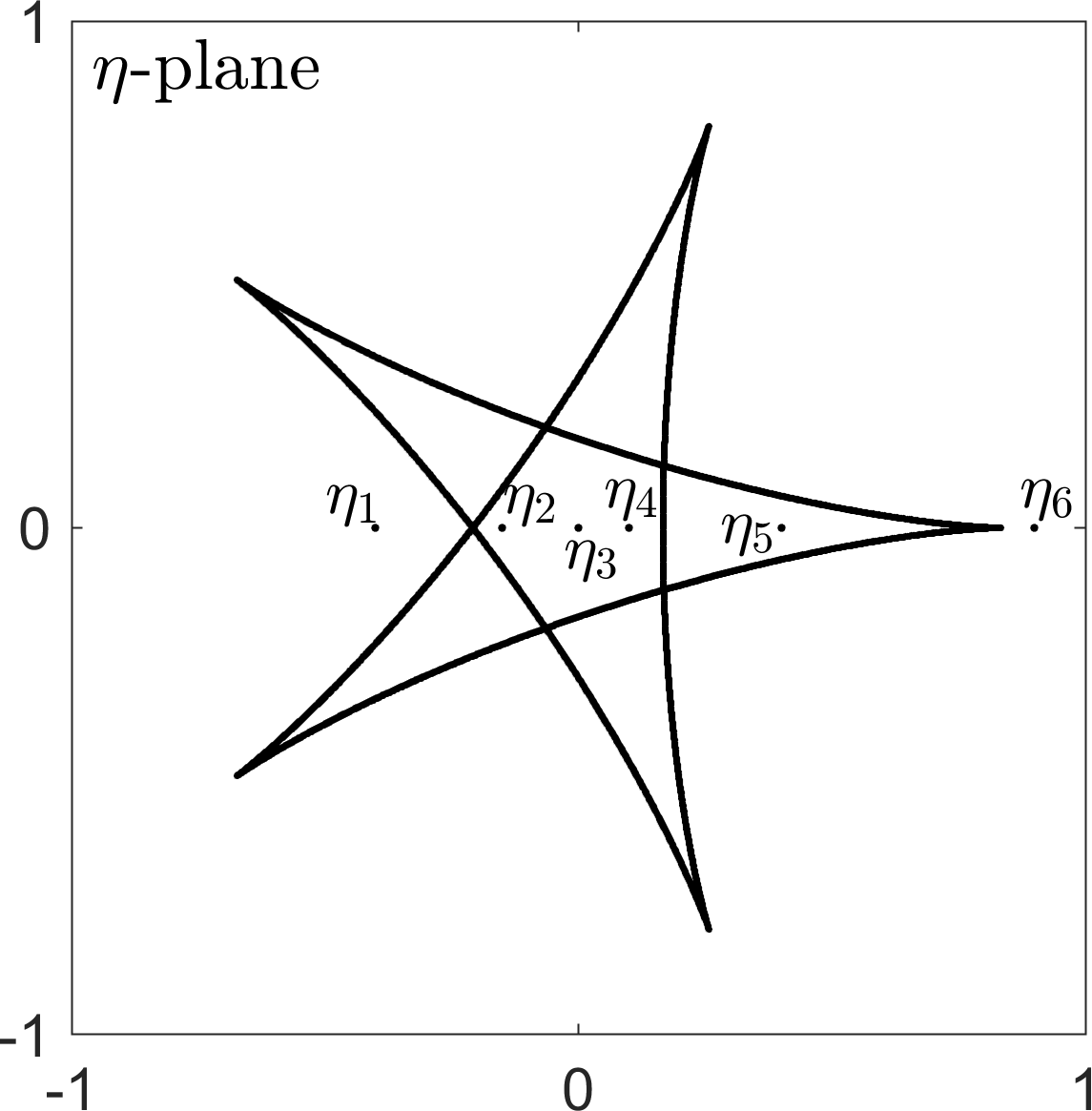}
\vspace{.2cm}

}
\caption{Caustics of $f(z) = \frac{1}{3}z^3 + \frac{1}{2}\conj z^2$; see 
Example~\ref{exp:N}.}
\label{fig:poly_caus}
\end{figure}

\begin{figure}[t!]
	{\centering
		\includegraphics[width=0.32\linewidth, height = 0.32\linewidth]{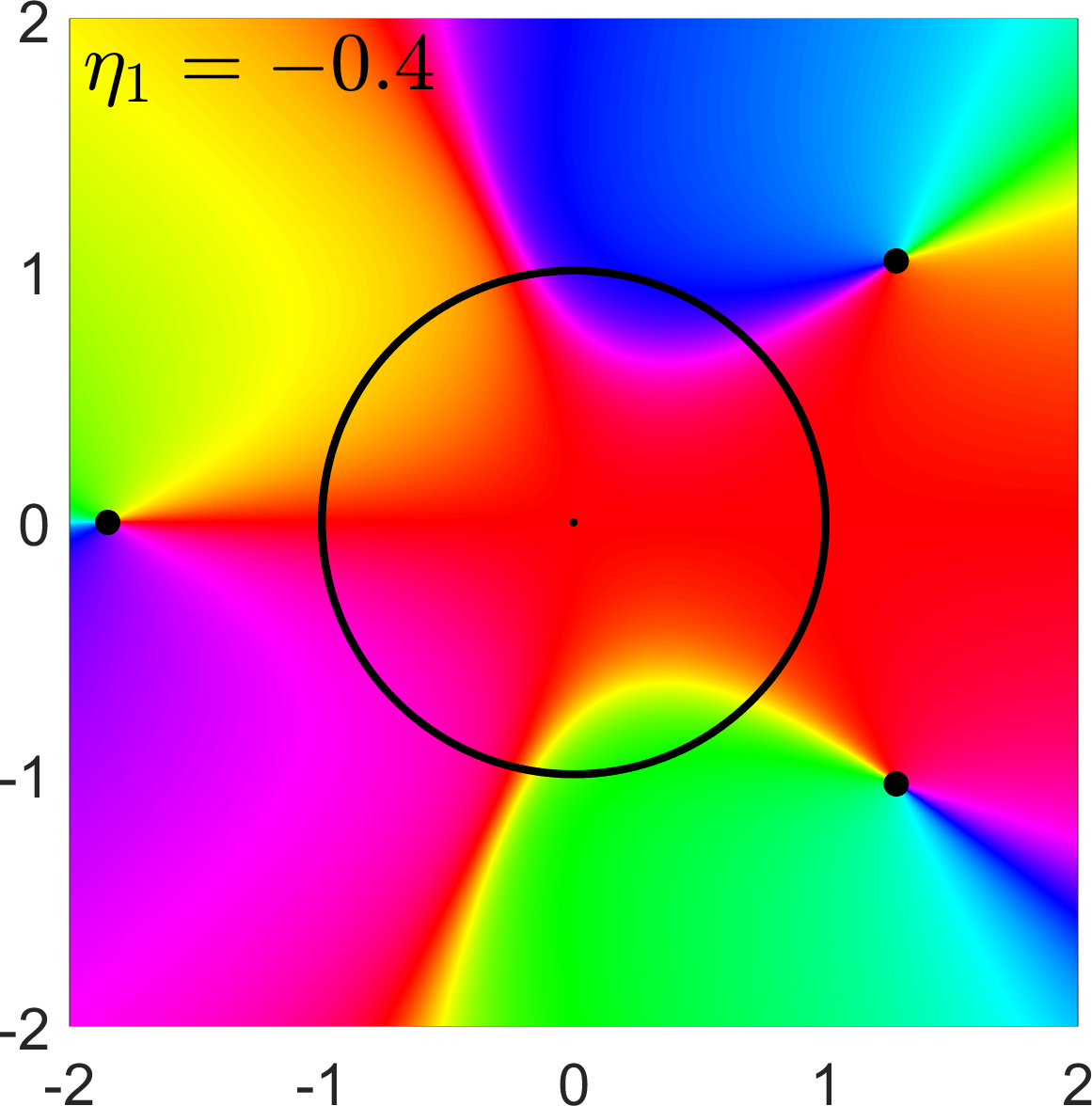}
		\includegraphics[width=0.32\linewidth, height = 0.32\linewidth]{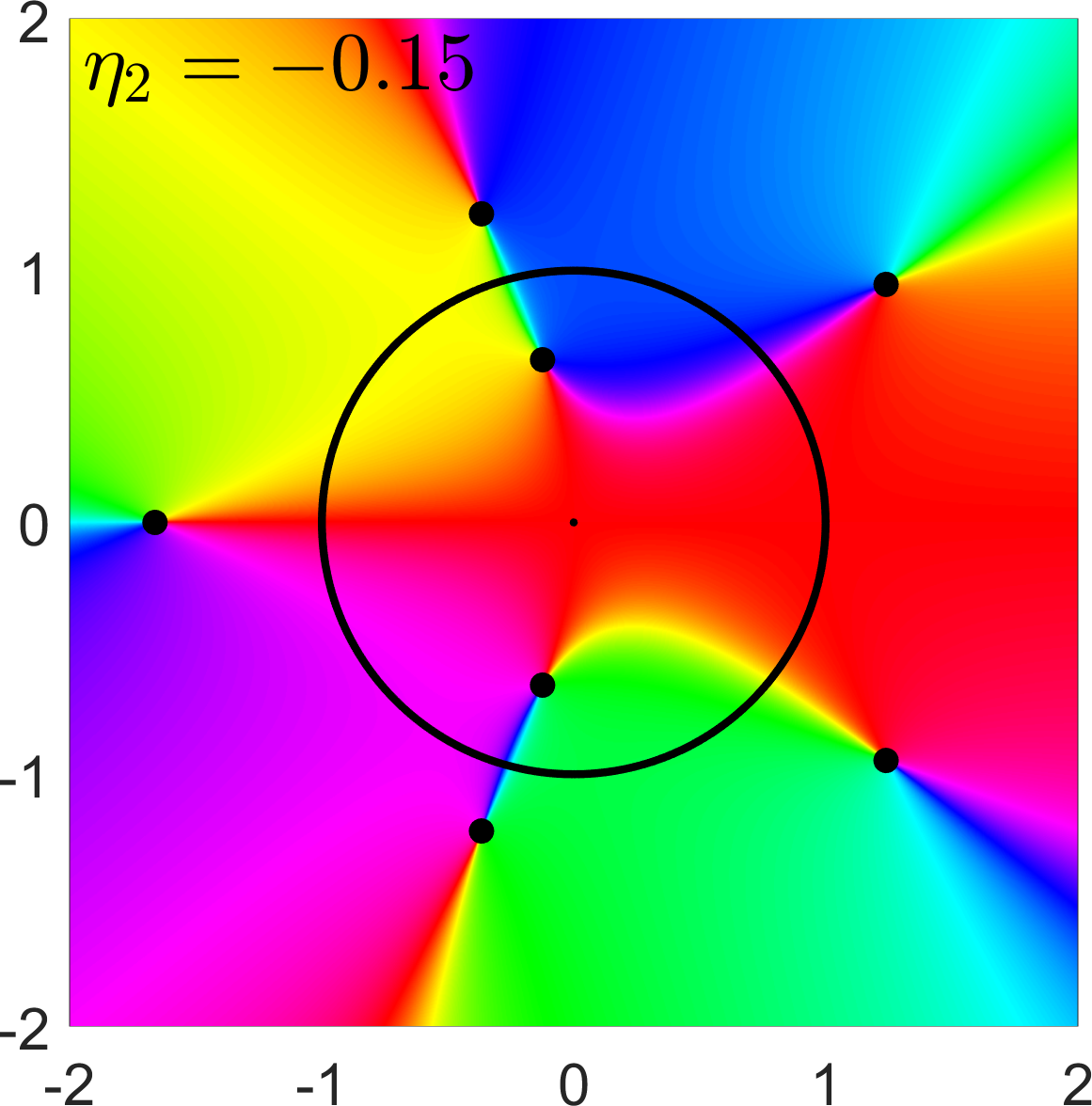}
		\includegraphics[width=0.32\linewidth, height = 0.32\linewidth]{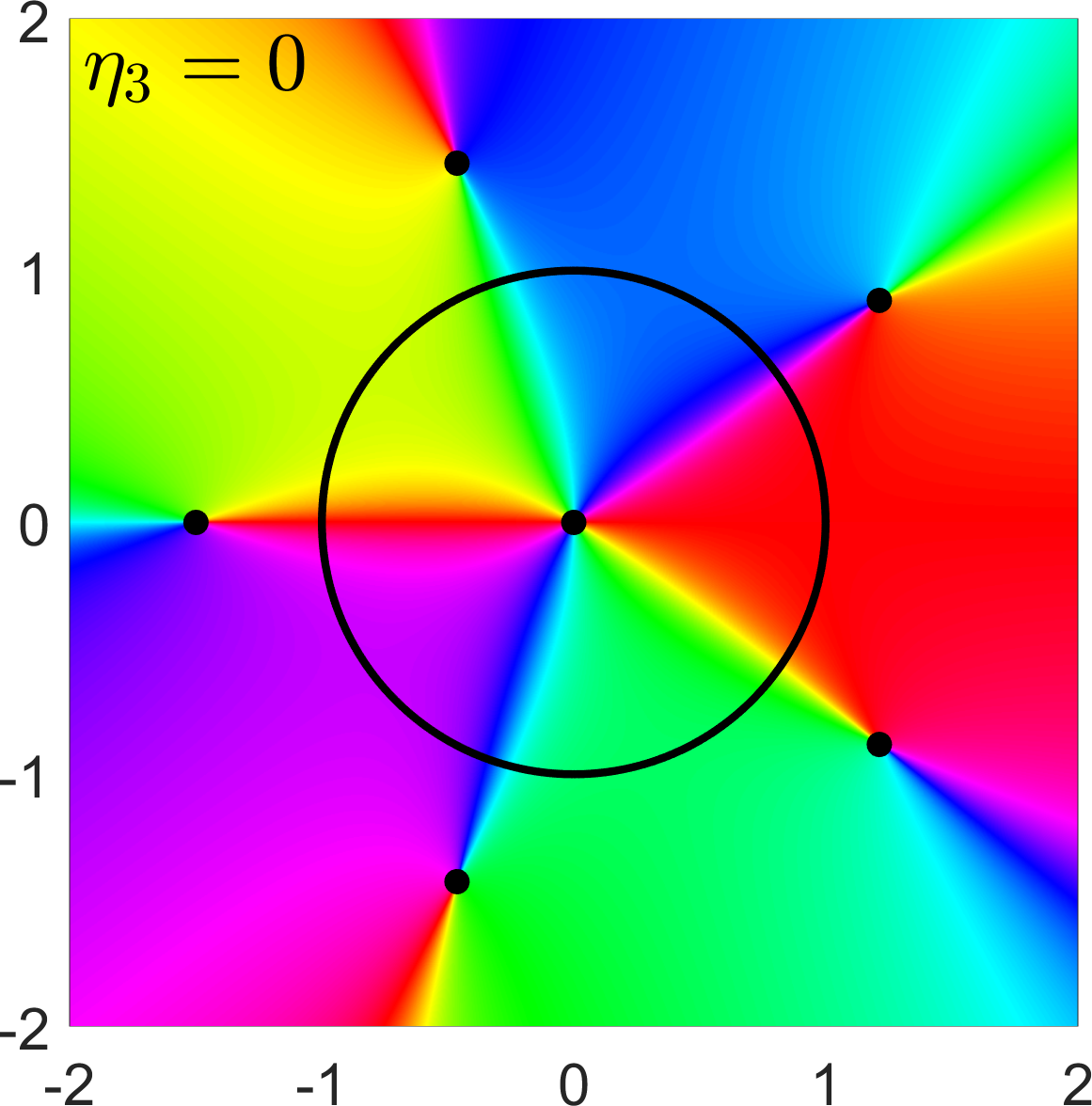}

}

	{\centering
		\includegraphics[width=0.32\linewidth, height = 0.32\linewidth]{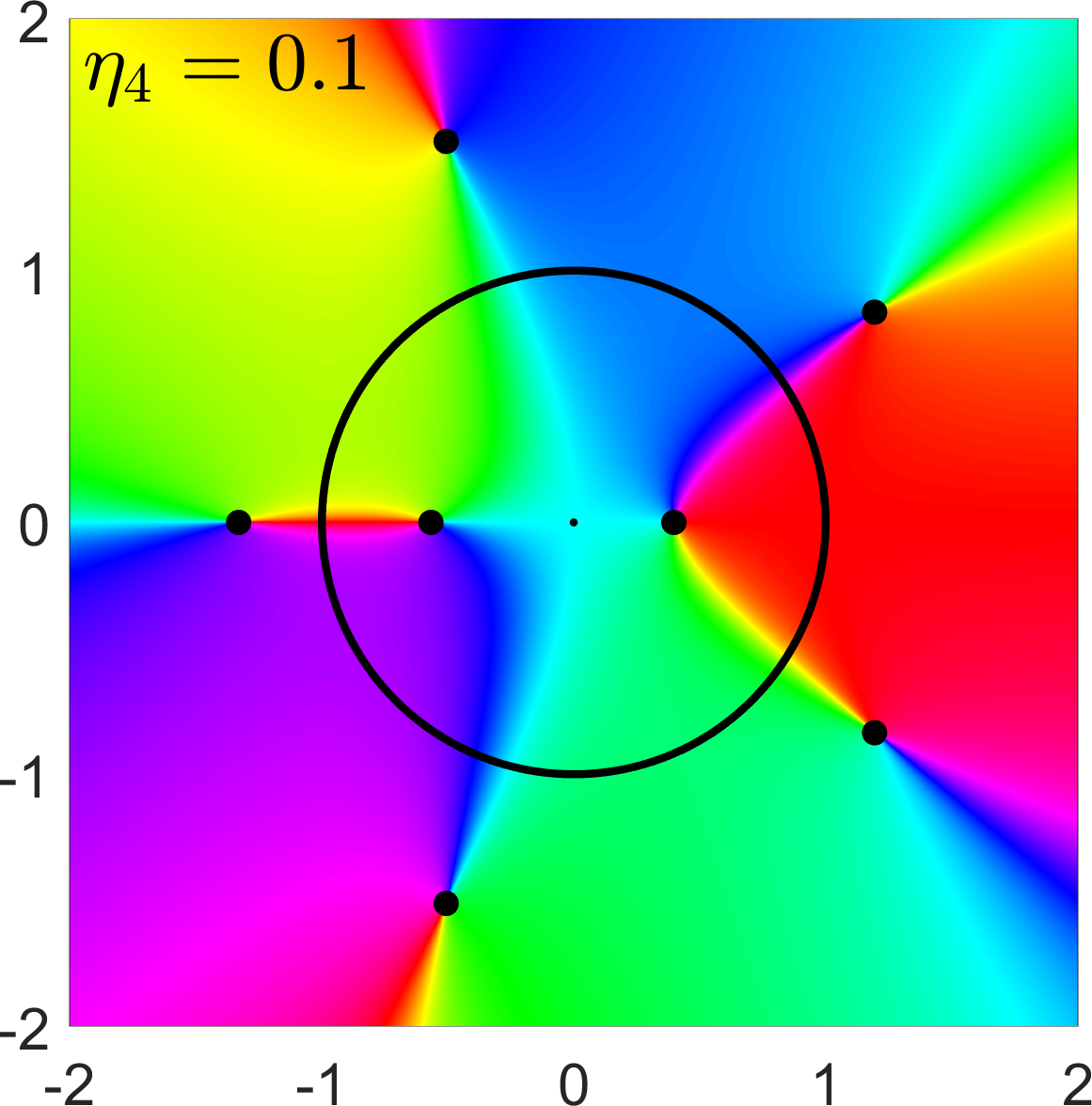}
		\includegraphics[width=0.32\linewidth, height = 0.32\linewidth]{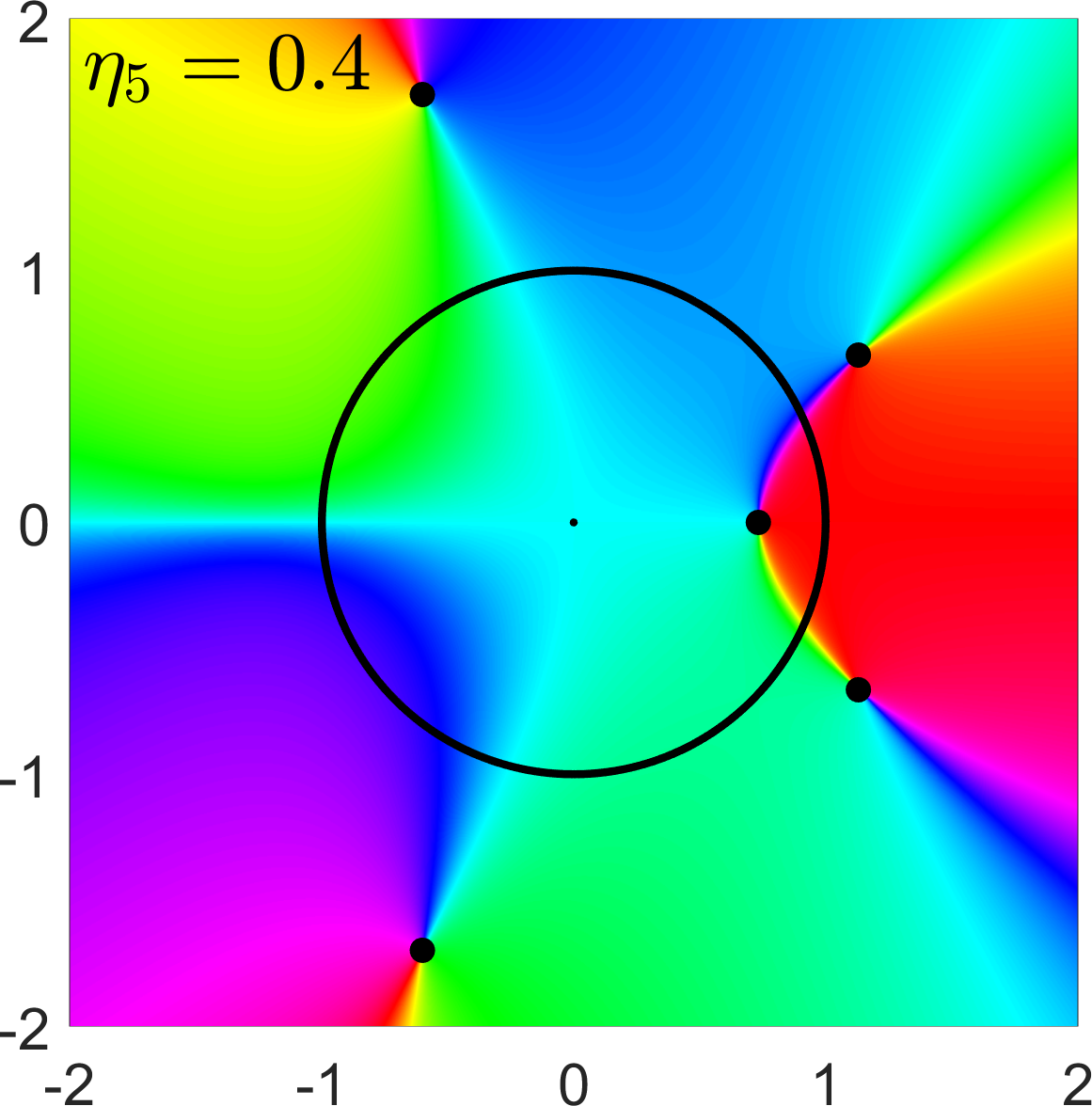}
		\includegraphics[width=0.32\linewidth, height = 0.32\linewidth]{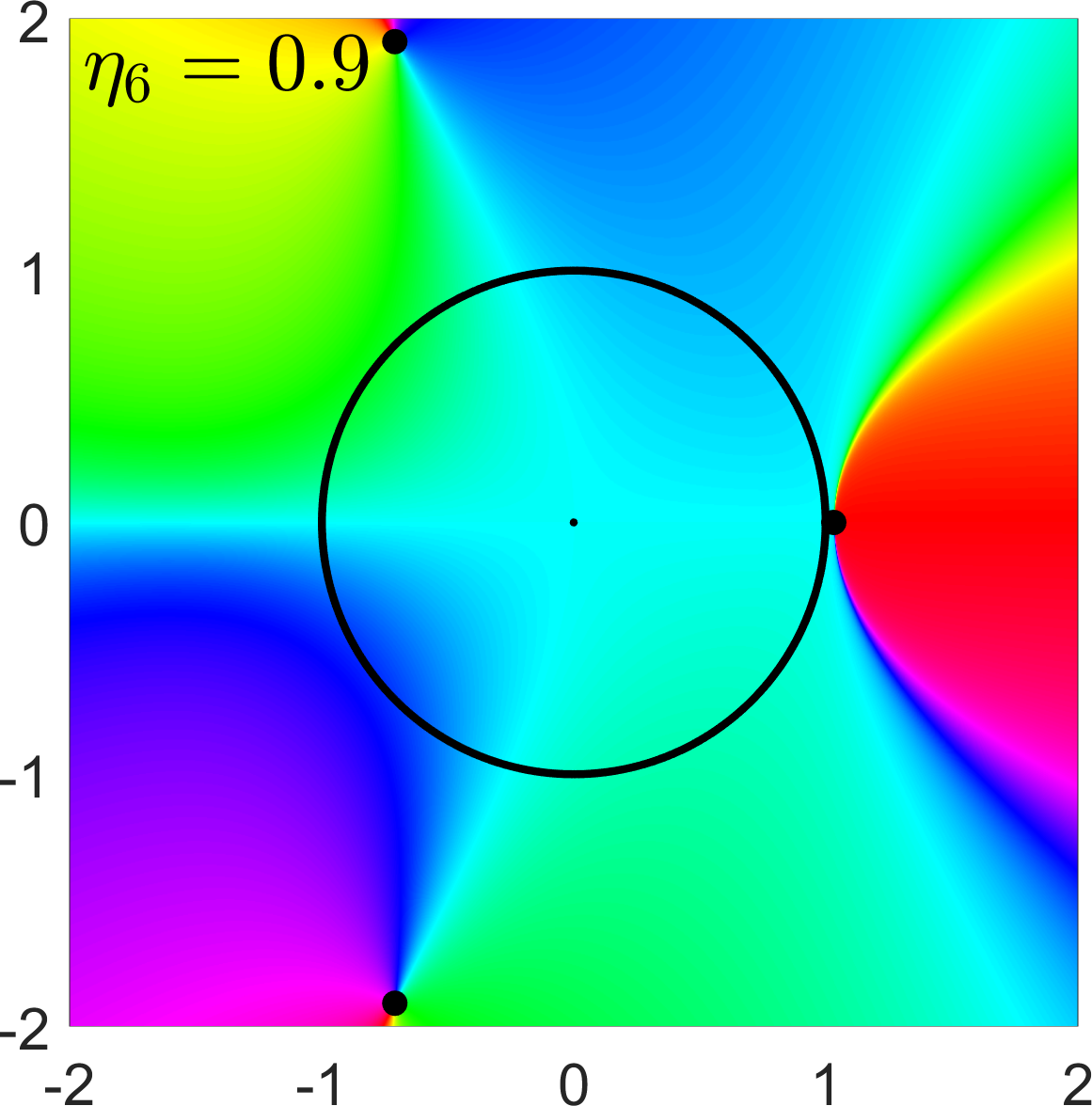}
		
}

\vspace{.2cm}

\caption{Phase plots of $f_{\eta_j}(z) = \frac{1}{3}z^3 + \frac{1}{2}\conj z^2
- \eta_j$ (see Figure~\ref{fig:poly_caus}).  Black dots indicate zeros of 
$f_{\eta_j}$.  The critical set $\cC = \partial \bD \cup \{0\}$ is displayed in 
black.}
\label{fig:N_exp}
\end{figure}

\section{On the number of zeros of harmonic polynomials}\label{sect:polynomials}

We consider harmonic polynomials
\begin{equation}\label{eqn:harmonic_polynomial}
f(z) = p(z) + \conj{q(z)} = \sum_{k=0}^n a_k z^k + \conj{\sum_{k=0}^m b_k 
z^k}, \quad n \geq m,
\end{equation}
with $a_n \neq 0 \neq b_m$.  These are non-degenerate if and only if $\abs{a_n} 
\neq \abs{b_n}$, where we define $b_n = 0$ for $n > m$.
Such functions have at most $n^2$ zeros and this bound is 
sharp~\cite{Wilmshurst1998}.
By the argument principle, $f$ has at least $n$ zeros, if none of them is 
singular. If $f$ has fewer than $n$ zeros, at least one has to be in $\cM$. 
However, counting the zeros with their Poincar\'e indices as multiplicities 
gives again at least $n$ zeros in total.

For $n > m \geq 1$, we study the \emph{maximum valence} of harmonic polynomials
\begin{equation*}
V_{n,m} = \max \{N(p(z) + \conj{q(z)}):\deg(p)=n, \deg(q) = m\},
\end{equation*}
where $N(f)$ denotes the number of zeros of $f$.
We have $V_{n,m} \leq n^2$ from~\cite{Wilmshurst1998}, but the quantity 
$V_{n,m}$ is only known in special cases, namely $V_{n,1} = 
3n-2$ from~\cite{KhavinsonSwiatek2003,Geyer2008} and $V_{n,n-1} = n^2$ 
from~\cite{Wilmshurst1998}.  We show in this section, that for given $n > m 
\geq 1$ and every $k \in \{ n, n+1, \ldots, V_{n,m} \}$, there exists a 
harmonic polynomial~\eqref{eqn:harmonic_polynomial} with $k$ zeros, i.e., every 
number of zeros between the lower and upper bound occurs.
This generalizes~\cite[Thm.~1.1]{BleherHommaJiRoeder2014}.
More precisely, we can achieve all these numbers by just changing $a_0$, 
which is equivalent to considering the pre-images of a certain $\eta$ instead 
of the zeros.

If $\eta$ crosses a single caustic arc at a fold, the number of pre-images 
changes by $\pm 1$ ($\eta$ on the caustic) and $\pm 2$ ($\eta$ on the ``other 
side'' of the caustic) by Theorems~\ref{thm:relative_counting} 
and~\ref{thm:local_fold}.
The key difficulty now is to handle multiple caustic arcs, 
i.e., caustic arcs which are the image of several different critical arcs.

\begin{example} \label{ex:multiple_caustic_arc}
Consider $f(z) = \frac{1}{2}p(z)^2 + \conj{p(z)}$ with $p(z) = z^2 - 1$.
Then $J_f(z) = \abs{p'(z)}^2 (\abs{p(z)}^2 - 1)$, and
$\cC = \{ z \in \C : \abs{p(z)} = 1\}$ consists of the two curves 
$\gamma_\pm(t) = \pm \sqrt{1 + e^{it}}$, $- \pi \leq t \leq \pi$.
Since $p(\gamma_+(t)) = 
p(\gamma_-(t))$, the harmonic mapping $f$ maps $\gamma_\pm$ onto the same 
caustic.

More generally, let $\gamma$ be a closed curve with $\abs{g'(w)/h'(w)} = 
1$ on $\tr(\gamma)$, and let $w = p(z)$ such that $\tr(\gamma)$ has $k \geq 2$ 
disjoint
pre-images under $p$.  Then these pre-images are in the critical set of
$f(z) = h(p(z)) + \conj{g(p(z))}$ and are mapped to the same caustic.
In particular, $h(z) = \frac{1}{n} z^n$, $g(z) = \frac{1}{m} z^m$ with $n > m 
\geq 1$ provides an example of a non-degenerate harmonic polynomial with $k$ 
critical curves that are mapped onto the same caustic.
\end{example}

Multiple caustic arcs can be eliminated by a polynomial perturbation of~$f$.
We write $\cC_f$ and $\cC_F$ for the critical sets of $f$ and $F$, respectively.

\begin{lemma}\label{lem:perturbing_caus}
Let $f$ be a harmonic mapping, and $z_1, z_2 \in \cC_f$, $z_1 \neq z_2$, with 
$f(z_1) = f(z_2)$. 
Then there exists a polynomial $p$ with $\deg(p) = 3$, such that $z_1, 
z_2 \in \cC_F$ for $F = f + p$, but $F(z_1) \neq F(z_2)$.
\end{lemma} 

\begin{proof}
Let $\eps > 0$, and let $p$ be the (unique) Hermite interpolation polynomial of 
degree $3$ with $p(z_1) = \eps$, $p(z_2) = -\eps$, and $p'(z_1) = 0 = p'(z_2)$. 
We then have $J_F (z_1) = 0 = J_{f} (z_1)$, and 
the same for $z_2$, but $F(z_1) \neq F(z_2)$.
\end{proof}

Next, we show that sufficiently small perturbations do not decrease the number 
of non-singular zeros.

\begin{lemma}\label{lem:perturbing_zeros}
Let $f$ and $g$ be harmonic mappings, such that $f$ has only finitely 
many zeros, which are all non-singular,
and such that $g$ has no singularities at the zeros of $f$.  Then
$N(f) \le N(f+\eps g)$ for all sufficiently small $\eps > 0$.
\end{lemma}

\begin{proof}
Let $z_1, \dots, z_n$ be the zeros of $f$. Since non-singular zeros are 
isolated~\cite[p.~413]{DurenHengartnerLaugesen1996}, there exists $\delta > 0$, 
such that
$D_\delta(z_j) \cap \cC = \emptyset$, $f$ and $g$ have no other exceptional 
points than $z_j$ in $\conj{D_\delta(z_j)}$ for $j=1,\dots,n$, and
$D_\delta(z_j) \cap D_\delta(z_k) = \emptyset$ for $j \neq k$.

Define $\Gamma = \cup_{k=1}^n \partial D_\delta(z_k)$ and let $\eps > 0$ such 
that
\begin{equation*}
\eps \cdot \max\{\abs{g(z)} : z \in \Gamma\}
< \min \{\abs{f(z)} : z \in \Gamma\}.
\end{equation*}
Then we have for $z  \in  \Gamma$
\begin{equation*}
\abs{f(z) - (f(z) + \eps g(z))} = \eps \abs{g(z)} < \abs{f(z)}.
\end{equation*}
By Rouch\'e's theorem (e.g.~\cite[Thm.~2.3]{SeteLuceLiesen2015a}) and 
the argument principle applied on each $\partial D_\delta(z_k)$, we get
\begin{equation*}
N(f) = \sum_{k=1}^n N(f; D_\delta(z_k)) \le  \sum_{k=1}^n N(f+\eps g; 
D_\delta(z_k)) \le N(f+\eps g),
\end{equation*}
which settles the proof.
\end{proof}

With the Lemmas~\ref{lem:perturbing_caus} and~\ref{lem:perturbing_zeros} we get 
the following result on the possible number of zeros of harmonic polynomials.

\begin{theorem}\label{thm:valence_polynomials}
Let $n > m \geq 1$ and $k \in \{ n, n+1, \ldots, V_{n,m} \}$.  Then
there exists a harmonic polynomial $f(z) = p(z) + \conj{q(z)}$ with 
$\deg(p) = n$ and $\deg(q) = m$, and with $k$ zeros.

Moreover, if $k$ and $n$ have different parity ($n-k$ is odd), then $f$ is 
singular, i.e., $0$ is a caustic point of $f$.
If $k$ and $n$ have the same parity, then there exists a non-singular $f$, as 
prescribed above.
\end{theorem}

\begin{proof}
Let $f(z) = p(z) + \conj{q(z)}$ be a harmonic polynomial with $\deg(p) = 
n$, $\deg(q) = m$, and with $V_{n,m}$ zeros, which exists by the 
definition of $V_{n,m}$.  Without loss of generality, we can assume that $f$ 
has no multiple caustic arcs.
Indeed, when
$n=2$ the only critical curve of $f$ is the image of the unit circle under 
a M\"obius transformation, and hence there are no multiple caustic arcs.
If $n \ge 3$ and if $f$ has multiple caustic arcs we resolve them by 
Lemma~\ref{lem:perturbing_caus} with a polynomial perturbation of degree $3$, 
such that no other multiple caustic arcs occur.
For sufficiently small $\eps > 0$, the resulting harmonic polynomial 
has at most $V_{n,m}$ zeros, and at least
$V_{n,m}$ zeros by Lemma~\ref{lem:perturbing_zeros}.
This gives a harmonic polynomial with $V_{n,m}$ zeros and without multiple 
caustic arcs.

By Theorem~\ref{thm:large_eta}, there exists an $\eta_n \in \C$ with
$N_{\eta_n}(f) = n$.  Let $\phi$ be a curve from $\eta_n$ to $0$, which intersects 
the caustics only in folds corresponding to a single caustic arc.
Such a curve exists since (possible) multiple caustic arcs are already 
resolved, and since the zeros of $\psi$ are isolated by 
Lemma~\ref{lem:tangent_to_caustic}. Note that $f$ is light since any $f - 
\eta$ has at most $n^2$ zeros.
Then by Theorems~\ref{thm:relative_counting} 
and~\ref{thm:local_fold}, all $k = n, n+1, \ldots, V_{n,m}$ appear as 
number of pre-images under $f$ for an appropriate $\eta_k \in \tr(\phi)$,
i.e., $N_{\eta_k}(f) = k$, and hence $f - \eta_k$ is a harmonic polynomial with 
$k$ zeros.

The second part follows from 
Theorem~\ref{thm:relative_counting} and the fact that $\eta_n$ can be chosen 
in $\C \setminus f(\cC)$; see Theorem~\ref{thm:large_eta}.
\end{proof}

\begin{remark}
Let $n > m \geq 1$. By the proof of 
Theorem~\ref{thm:valence_polynomials}, there 
exists a harmonic polynomial $f(z) = p(z) + \conj{q(z)}$ with $\deg(p) = n$, 
$\deg(q) = m$, and $\eta_n, \ldots$, $\eta_{V_{n,m}} \in \C$, such that $f - 
\eta_k$ has $k$ zeros. Moreover, $\eta_{n+1}, \eta_{n+3}, \ldots$ are on the 
caustics of $f$, and $\eta_n, \eta_{n+2}, \ldots$ can be chosen in caustic 
tiles.
\end{remark}

Since $V_{n,n-1} = n^2$, we have the following corollary.

\begin{corollary}
Let $n \geq 2$. For each $k \in \{ n, n+1, \ldots, n^2 \}$, there exists 
a harmonic 
polynomial as in~\eqref{eqn:harmonic_polynomial} with $k$ zeros.
\end{corollary}

\section{Outlook}

A further study of the geometry of the caustics should be of interest, e.g., the number of cusps.  This 
an important open problem posed by Petters~\cite[p.~1399]{Petters2010} for certain 
harmonic mappings from gravitational lensing.

While we considered harmonic mappings on the Riemann sphere 
(minus possible poles) in 
this work, also harmonic mappings in bounded domains (similar 
to~\cite{Neumann2005}) and on more general Riemann surfaces might be of 
interest.
We expect similar results for these domains of definition.

The results in Section~\ref{sect:polynomials} could probably be generalized to 
a broader class of harmonic mappings, e.g., non-degenerate rational harmonic 
mappings $f(z) = r(z) + \conj{s(z)}$, using the same approach as above.  
However, one would have to handle multiple caustic arcs in a different way.

\paragraph*{Acknowledgments.}
We thank J\"org Liesen for several helpful comments on the manuscript.
Moreover, we are grateful to the anonymous referees for many valuable comments, 
which lead to improvements of this work.

\footnotesize
\bibliography{valence}
\bibliographystyle{siam} 

\end{document}